\documentclass[11pt,a4paper]{article}
\usepackage[a4paper]{geometry}
\usepackage{amssymb,latexsym,amsmath,amsfonts,amsthm}
\usepackage{graphicx}
\usepackage{algorithm}
\usepackage{algorithmic}
\usepackage{dsfont}
\usepackage{mathrsfs}
\usepackage{enumerate}
\usepackage{epsfig}
\usepackage{booktabs}
\usepackage{mathtools}
\usepackage{multirow}
\usepackage{pifont}
\usepackage{comment,verbatim}
\usepackage{overpic}
\usepackage{hyperref}
\usepackage{bbm}
\usepackage[toc,page]{appendix}

\newtheorem{theorem}{Theorem}[section]
\newtheorem{lemma}[theorem]{Lemma}

\theoremstyle{definition}
\newtheorem{definition}[theorem]{Definition}

\theoremstyle{remark}

\newtheorem{remark}[theorem]{Remark}

\numberwithin{equation}{section}

\usepackage{color}

\usepackage[normalem]{ulem}

\def\XXint#1#2#3{{
\setbox0=\hbox{$#1{#2#3}{\int}$}
\vcenter{\hbox{$#2#3$}}\kern-.5\wd0}}

\begin{document}
\title{Optimal rates of convergence and error localization of Gegenbauer projections}
\author{Haiyong Wang\footnotemark[1]~\footnotemark[2]}
\maketitle
\renewcommand{\thefootnote}{\fnsymbol{footnote}}
\footnotetext[1]{School of Mathematics and Statistics, Huazhong
University of Science and Technology, Wuhan 430074, P. R. China.
E-mail: \texttt{haiyongwang@hust.edu.cn}}

\footnotetext[2]{Hubei Key Laboratory of Engineering Modeling and
Scientific Computing, Huazhong University of Science and Technology,
Wuhan 430074, P. R. China.}

\begin{abstract}

Motivated by comparing the convergence behavior of Gegenbauer
projections and best approximations, we study the optimal rate of
convergence for Gegenbauer projections in the maximum norm. We show
that the rate of convergence of Gegenbauer projections is the same
as that of best approximations under conditions of the underlying
function is either analytic on and within an ellipse and
$\lambda\leq0$ or differentiable and $\lambda\leq1$, where $\lambda$
is the parameter in Gegenbauer projections. If the underlying
function is analytic and $\lambda>0$ or differentiable and
$\lambda>1$, then the rate of convergence of Gegenbauer projections
is slower than that of best approximations by factors of
$n^{\lambda}$ and $n^{\lambda-1}$, respectively. An exceptional case
is functions with endpoint singularities, for which Gegenbauer
projections and best approximations converge at the same rate for
all $\lambda>-1/2$. For functions with interior or endpoint
singularities, we provide a theoretical explanation for the error
localization phenomenon of Gegenbauer projections and for why the
accuracy of Gegenbauer projections is better than that of best
approximations except in small neighborhoods of the critical points.
Our analysis provides fundamentally new insight into the power of
Gegenbauer approximations and related spectral methods.

\end{abstract}


\vspace{0.05in}


\section{Introduction}\label{sec:introduction}
Orthogonal polynomials are ubiquitous in approximation theory and
numerical analysis and play crucial roles in numerous applications,
including the construction of Gaussian quadrature (Davis \&
Robinowitz, 1984), the resolution of Gibbs phenomenon (Adcock \&
Hansen, 2012; Gelb \& Tanner, 2006; Gottlieb \& Shu, 1997), and
spectral methods for the numerical solution of differential
equations (Guo, 2000; Hesthaven \emph{et al.}, 2007; Olver \&
Townsend, 2013; Shen \emph{et al.}, 2011). One of the most
attractive features of orthogonal polynomials is that their
approximation power depends solely on the regularity of the
underlying function and hence fast convergence can be achieved
whenever the underlying function is sufficiently smooth. Due to the
important role that orthogonal polynomials plays in diverse areas of
mathematic and physics, their approximation properties have
attracted considerable interest, especially in the spectral methods
community (e.g., Canuto, \emph{et al.}, 2006; Hesthaven \emph{et
al.}, 2007; Shen \emph{et al.}, 2011; Trefethen, 2013).

Let $\mathrm{d}\mu$ be a positive Borel measure on the interval
$[a,b]$, for which all moments of $\mathrm{d}\mu$ are finite. We
introduce the inner product $\langle f,g \rangle_{\mathrm{d}\mu} =
\int_{a}^{b} f(x) g(x) \mathrm{d}\mu(x)$ and let
$\{\varphi_k\}_{k=0}^{\infty}$ be a set of orthogonal polynomials
with respect to $\mathrm{d}\mu$. Then, for any $f\in L^2([a,b])$, it
can be expanded in terms of $\{\varphi_k\}$ as
\begin{align}\label{eq:FourierSeries}
f(x) = \sum_{k=0}^{\infty} f_k \varphi_k(x), \quad  f_k =
\frac{\langle f,\varphi_k \rangle_{\mathrm{d}\mu}}{\langle
\varphi_k,\varphi_k \rangle_{\mathrm{d}\mu}}.
\end{align}
Let $S_n(f)$ denote the truncation of the infinite series above
after the first $n+1$ terms, i.e., $S_n(f) = \sum_{k=0}^{n} f_k
\varphi_k(x)$, it is well known that $S_n(f)$ is the orthogonal
projection of $f$ onto the space
$\mathcal{P}_{n}=\mathrm{span}\{1,x,\ldots,x^{n}\}$. Existing
approaches for analyzing convergence of $S_n(f)$ in the maximum norm
can be roughly categorized into two types: (i) applying the
Lebesgue's lemma $\|f - S_n(f) \|_{\infty} \leq (1+\Lambda)\| f -
\mathcal{B}_n(f) \|_{\infty}$, where $\Lambda =
\sup_{f\not\equiv0}\|S_n(f)\|_{\infty}/\|f\|_{\infty}$ is the
Lebesgue constant of $S_n(f)$ and $\mathcal{B}_n(f)$ is the best
polynomial approximation of degree $n$ to $f$, i.e.,
$\|f-\mathcal{B}_n(f)\|_{\infty}=\min_{p\in\mathcal{P}_{n}}\|f-p\|_{\infty}$.
Hence, this approach transforms the error estimate of $S_n(f)$ to
the problem of finding estimates for the corresponding Lebesgue
constant; (ii) using the inequality $\|f - S_n(f) \|_{\infty} \leq
\sum_{k=n+1}^{\infty} |f_k| \|\varphi_k\|_{\infty}$, and the
remaining task is to find some sharp estimates of the coefficients
$\{f_k\}$. The former approach plays a key role in analyzing uniform
convergence of orthogonal projections and nowadays estimates for the
Lebesgue constants associated with classical orthogonal projections
have been well-understood. However, as far as we are aware, the
sharpness of the predicted convergence rates has not been addressed.
For the latter approach, a remarkable advantage is that some
computable error bounds of $S_n(f)$ can be established (e.g.,
Bernstein, 1912; Liu \emph{et al.}, 2019; Liu \emph{et al.}, 2021;
Trefethen, 2013; Wang \& Xiang, 2012; Wang, 2018; Wang, 2021; Xiang,
2012; Xiang \& Liu, 2020; Zhao \emph{et al.}, 2013). However, as
shown in Wang (2018) and Wang (2021), the convergence rate predicted
by this approach may be slower than the actual convergence rate.

In this work we are concerned with optimal rates of convergence of
Gegenbauer projections in the maximum norm, i.e.,
$\mathrm{d}\mu(x)=(1-x^2)^{\lambda-1/2}\mathrm{d}x$, where
$\lambda>-1/2$ and $[a,b]=[-1,1]$. In order to exhibit the
dependence on the parameter $\lambda$, we denote by
$S_n^{\lambda}(f)$ the Gegenbauer projection of degree $n$. By
Lebesgue's lemma, we have
\begin{align}\label{eq:errorG}
\|f - S_n^{\lambda}(f) \|_{\infty} &\leq (1 + \Lambda_n(\lambda)) \|
f - \mathcal{B}_n(f) \|_{\infty},
\end{align}
where
$\Lambda_n(\lambda)=\sup_{f\not\equiv0}\|S_n^{\lambda}(f)\|_{\infty}/\|f\|_{\infty}$
is the Lebesgue constant of Gegenbauer projections. It is known from
(Frenzen \& Wong, 1986; Levesley \& Kushpel, 1999; Lorch, 1959) that
\begin{align}\label{eq:Lebesgue}
\Lambda_n(\lambda) = \left\{
\begin{array}{ll}
{\displaystyle O(n^{\lambda}) }, & \hbox{$\lambda>0$,}   \\[8pt]
{\displaystyle O(\log n) },      & \hbox{$\lambda=0$,}   \\[8pt]
{\displaystyle O(1) },           & \hbox{$\lambda<0$.}
            \end{array}
            \right.
\end{align}
Note that the inequality (1.2) holds true for all $f\in{C}[-1,1]$.
One might ask how sharp the error estimates for $S_n^{\lambda}(f)$
obtained above are. First, it is easily seen that the predicted rate
of convergence of $S_n^{\lambda}(f)$ is optimal in the case
$\lambda<0$ since it is the same as that of $\mathcal{B}_n(f)$, and
is near-optimal in the case $\lambda=0$ since the Lebesgue constant
$\Lambda_n(\lambda)$ grows very slowly as $n$ increases. In the case
$\lambda>0$, we see that the rate of convergence of
$S_n^{\lambda}(f)$ is slower than that of $\mathcal{B}_n(f)$ by at
most a factor of $n^{\lambda}$. This difference may be negligible
for functions which are analytic in a region containing the interval
$[-1,1]$, but will be crucial for functions which are only
continuously differentiable on the interval $[-1,1]$. More recently,
the particular case of $\lambda=1/2$, which corresponds to Legendre
projections, was examined in Wang (2021). It was shown that the
predicted rate of convergence by \eqref{eq:errorG} is sharp, up to
constant factors, whenever the underlying function is analytic, but
is slower than the actual rate of convergence whenever the
underlying function is differentiable, such as piecewise analytic
functions of class $C^{s}[-1,1]$ with $s$ being a nonnegative
integer (see Definition \ref{def:PiecewiseAnal}) and functions with
algebraic singularities. Further, it was shown that the convergence
rates of Legendre projections for these differentiable functions are
actually the same as that of $\mathcal{B}_n(f)$. In this
perspective, it will be interesting to continue in this direction
and explore the case of Gegenbauer projections.

We highlight the main contributions of this paper as follows.
\begin{itemize}
\item[(i)] If $f$ is analytic in the region bounded by the
ellipse with foci $\pm1$ and the sum of the semiminor and semimajor
axes is $\rho>1$, we improve the existing results in Wang (2016) and
establish some new explicit error bounds for $S_n^{\lambda}(f)$. We
show that the inequality \eqref{eq:errorG} is sharp in the sense
that the convergence rate of $\mathcal{B}_n(f)$ is better than that
of $S_n^{\lambda}(f)$ by a factor of $n^{\lambda}$ for $\lambda>0$.

\item[(ii)] If $f$ belongs to the space of piecewise analytic functions
of class $C^{m-1}[-1,1]$ for some $m\in\mathbb{N}$, we establish
optimal convergence rates for $S_n^{\lambda}(f)$ and show that the
predicted rate of convergence by the inequality \eqref{eq:errorG} is
slower than the actual rate of convergence by a factor of
$n^{\min\{\lambda,1\}}$ whenever $\lambda>0$.

\item[(iii)] If $f$ has an interior or endpoint algebraic singularity, we
carry out a convergence analysis of $S_n^{\lambda}(f)$ for the model
function $f(x)=|x-\theta|^{\alpha}$, where $\theta\in[-1,1]$ and
$\alpha>0$ is not an even integer whenever $\theta\in(-1,1)$ and is
not an integer whenever $\theta=\pm1$. In the case of
$\theta\in(-1,1)$, we show that the maximum error of
$S_n^{\lambda}(f)$ is attained at one of the {\it critical points}
(i.e., $x=-1,\theta,1$), and the predicted rate of convergence by
the inequality \eqref{eq:errorG} is slower than the actual rate of
convergence by a factor of $n^{\min\{\lambda,1\}}$ for $\lambda>0$.
In the case of $\theta=\pm1$, we show that the maximum error of
$S_n^{\lambda}(f)$ is attained at $x=\theta$ and the predicted rate
of convergence by the inequality \eqref{eq:errorG} in this case is
slower than the actual rate of convergence by a factor of
$n^{\lambda}$ for all $\lambda>0$.

\item[(iv)] We derive pointwise rates of convergence of
$S_n^{\lambda}(f)$ for the model function defined above and show
that the convergence rate of $S_n^{\lambda}(f)$ at each point
$x\in(-1,\theta)\cup(\theta,1)$ is faster than that of at
$x=\theta$. As a consequence, we explain not only the error
localization property of $S_n^{\lambda}(f)$, i.e., the error away
from the singularity is smaller than the error at the singularity,
but also why the accuracy of $S_n^{\lambda}(f)$ is better than that
of $\mathcal{B}_n(f)$ except in small neighborhoods of critical
points.
\end{itemize}

The paper is organized as follows. In the next section, we introduce
some preliminaries which will be useful in the sequel. In section
\ref{sec:experiment}, we carry out numerical experiments on the
convergence rates of $S_n^{\lambda}(f)$ and $\mathcal{B}_n(f)$ and
then give some observations. In section \ref{sec:Analytic}, we
establish explicit error bounds of $S_n^{\lambda}(f)$ for analytic
functions. We analyze optimal rates of convergence of
$S_n^{\lambda}(f)$ for piecewise analytic functions of class
$C^{m-1}[-1,1]$, where $m\in\mathbb{N}$, in section
\ref{sec:Piecewise} and for functions with algebraic singularities
in section \ref{sec:fractional}. Finally, we give some concluding
remarks in section \ref{sec:conclusion}.

\section{Preliminaries}\label{sec:prelim}
In this section, we introduce some basic properties of Gegenbauer
polynomials and the gamma function that will be used throughout the
paper. All these properties can be found in (Olver \emph{et al.},
2010; Szeg\H{o}, 1939).

\subsection{Gamma function}
For $\Re(z)>0$, the gamma function is defined by
\begin{align}
\Gamma(z) = \int_{0}^{\infty} t^{z-1} e^{-t} \mathrm{d}t.
\end{align}
When $\Re(z)\leq0$, $\Gamma(z)$ is defined by analytic continuation.
The gamma function satisfies the recursive property $ \Gamma(z+1) =
z \Gamma(z)$, and the classical reflection formula
\begin{align}\label{eq:reflection}
\Gamma(z) \Gamma(1-z) = \frac{\pi}{\sin(\pi z)}, \quad
z\neq0,\pm1,\ldots.
\end{align}
Moreover, the duplication formula of the gamma function reads
\begin{align}\label{eq:duplication}
\Gamma(2z) = \pi^{-1/2} 2^{2z-1} \Gamma(z) \Gamma\left(z +
\frac{1}{2}\right), \quad 2z\neq0,-1,-2,\ldots.
\end{align}
The ratio of two gamma functions will be crucial for the derivation
of explicit bounds for the Gegenbauer coefficients and the
asymptotic behavior of the reproducing kernel of Gegenbauer
projections. Let $a,b$ be some real or complex and bounded
constants, then we have
\begin{align}\label{eq:AsyGAMMA}
\frac{\Gamma(z+a)}{\Gamma(z+b)} = z^{a-b} \left[1 +
\frac{(a-b)(a+b-1)}{2z} + O(z^{-2}) \right], \quad
z\rightarrow\infty.
\end{align}
In the special case of either $a=1$ or $b=1$, the following sharp
bounds will be useful in the subsequent analysis.
\begin{lemma}\label{lem:gamma}
For $\gamma>-1$, it holds for every $k\in\mathbb{N}$ that
\begin{align}\label{eq:ratio1}
\frac{\Gamma(k+1)}{\Gamma(k+\gamma)} \leq k^{1-\gamma} \left\{
\begin{array}{ll}
{\displaystyle \frac{1}{\Gamma(1+\gamma)} }, & \hbox{$0\leq \gamma < 1$,}   \\[8pt]
{\displaystyle 1},    & \hbox{$-1<\gamma<0$~\text{\rm
or}~$\gamma\geq1$,}
            \end{array}
            \right.
\end{align}
and
\begin{align}\label{eq:ratio2}
\frac{\Gamma(k+\gamma)}{\Gamma(k+1)} \leq k^{\gamma-1} \left\{
\begin{array}{ll}
{\displaystyle 1}, & \hbox{$0\leq \gamma < 1$,}   \\[8pt]
{\displaystyle \Gamma(1+\gamma)},    & \hbox{$-1<\gamma<0$~{\rm
or}~$\gamma\geq1$.}
            \end{array}
            \right.
\end{align}
Moreover, these upper bounds in \eqref{eq:ratio1} and
\eqref{eq:ratio2} are sharp in the sense that they can be attained
either $k=1$ or $k=\infty$.
\end{lemma}
\begin{proof}
We only prove \eqref{eq:ratio1} and the proof of \eqref{eq:ratio2}
is completely analogous. In the cases $\gamma=0$ and $\gamma=1$,
\eqref{eq:ratio1} is trivial. Now consider the cases $-1<\gamma<0$
and $\gamma>0$ and $\gamma\neq1$. To this end, we introduce the
following sequence
\[
\psi(k) = \frac{\Gamma(k+1)}{\Gamma(k+\gamma)} k^{\gamma-1}.
\]
In view of the recursive property of $\Gamma(z)$, we obtain
\begin{align}
\frac{\psi(k+1)}{\psi(k)} = \frac{k+1}{k+\gamma} \left(\frac{k+1}{k}
\right)^{\gamma-1}. \nonumber
\end{align}
By differentiating the right-hand side of the above equation with
respect to $k$, one can easily check that the sequence
$\{\psi(k+1)/\psi(k)\}_{k=1}^{\infty}$ is strictly increasing
whenever $0<\gamma<1$ and is strictly decreasing whenever either
$-1<\gamma<0$ or $\gamma>1$. Since $\lim_{k\rightarrow\infty}
\psi(k+1)/\psi(k)=1$, we deduce that $\{\psi(k)\}_{k=1}^{\infty}$ is
strictly decreasing whenever $0<\gamma<1$ and is strictly increasing
whenever either $-1<\gamma<0$ or $\gamma>1$. Hence, for
$0<\gamma<1$, we have
\begin{align}
\psi(k)\leq\psi(1) ~~~ \Longrightarrow ~~~
\frac{\Gamma(k+1)}{\Gamma(k+\gamma)} \leq
\frac{k^{1-\gamma}}{\Gamma(1+\gamma)}, \nonumber
\end{align}
and the upper bound can be attained when $k=1$. For either
$-1<\gamma<0$ or $\gamma>1$, then
\begin{align}
\psi(k)\leq \lim_{k\rightarrow\infty} \psi(k) = 1 ~~~
\Longrightarrow ~~~  \frac{\Gamma(k+1)}{\Gamma(k+\gamma)} \leq
k^{1-\gamma}, \nonumber
\end{align}
and the upper bound can be attained when $k=\infty$. This proves
\eqref{eq:ratio1} and the proof of Lemma \ref{lem:gamma} is
complete.
\end{proof}

\subsection{Gegenbauer polynomials}
Let $n\geq0$ be an integer and let $\Omega:=[-1,1]$. The Gegenbauer
polynomial of degree $n$ is defined by
\begin{align}\label{def:GegenPoly}
C_{n}^{\lambda}(x) = \frac{(2\lambda)_n}{n!} {}_2\mathrm{
F}_1\left[\begin{matrix} -n, & n+2\lambda
\\   \lambda+\frac{1}{2}  \hspace{-1cm} &\end{matrix} ; ~ \frac{1-x}{2}
\right],
\end{align}
where ${}_2 \mathrm{F}_1(\cdot)$ is the Gauss hypergeometric
function defined by
\[
{}_2 \mathrm{F}_1 \left[\begin{matrix} a,~ b& \\  c
\end{matrix} \hspace{-.25cm} ;  z \right] = \sum_{k = 0}^{\infty} \frac{ (a)_k
(b)_k }{ (c)_k } \frac{ z^k }{ k! },
\]
and where $(z)_k$ denotes the Pochhammer symbol defined by $(z)_{k}
=(z)_{k-1} (z + k - 1)$ for $k\in\mathbb{N}$ and $(z)_0 = 1$. The
sequence of Gegenbauer polynomials
$\{C_k^{\lambda}(x)\}_{k=0}^{\infty}$ forms a system of polynomials
orthogonal over $\Omega$ with respect to the weight function
$\omega_{\lambda}(x)=(1-x^2)^{\lambda-1/2}$ and
\begin{equation}\label{eq:GegenOrthog}
\int_{\Omega} \omega_{\lambda}(x) C_{m}^{\lambda}(x)
C_{n}^{\lambda}(x) \mathrm{d}x = h_n^{\lambda}  \delta_{mn},
\end{equation}
where $\delta_{mn}$ is the Kronecker delta and
\[
h_n^{\lambda} = \frac{\pi 2^{1-2\lambda} \Gamma(n+2\lambda)}{
\Gamma(\lambda)^2 (n+\lambda) n!}, \quad \lambda>-1/2,~~
\lambda\neq0.
\]
Since $\omega_{\lambda}(x)$ is even, it follows that
$C_{n}^{\lambda}(x)$ satisfies the symmetry relation, i.e.,
$C_{n}^{\lambda}(x)=(-1)^n C_{n}^{\lambda}(-x)$ for each
$n=0,1,\ldots$, and this implies that $C_{n}^{\lambda}(x)$ is an
even function for even $n$ and an odd function for odd $n$. The
Rodrigues formula of Gegenbauer polynomials reads
\begin{align}\label{eq:Rodrigues}
\omega_{\lambda}(x) C_n^{\lambda}(x) =
\frac{-2\lambda}{n(n+2\lambda)} \frac{\mathrm{d}}{\mathrm{d}x}
\left\{ \omega_{\lambda+1}(x) C_{n-1}^{\lambda+1}(x) \right\},
\end{align}
which will be used in the asymptotic analysis of the Gegenbauer
coefficients.

Next, we state some explicit bounds on the maximum value of
Gegenbauer polynomials, which will be employed frequently in the
convergence analysis of Gegenbauer projections.
\begin{lemma}\label{lem:Bound}
If $\lambda>0$, then for all $n\in\mathbb{N}$,
\begin{align}\label{eq:BoundI}
\max_{|x|\leq1}|C_{n}^{\lambda}(x)| \leq n^{2\lambda-1} \left\{
\begin{array}{ll}
{\displaystyle \frac{1}{\Gamma(2\lambda)}},     & \hbox{$0<\lambda<1/2$,}   \\[10pt]
{\displaystyle 2\lambda},   & \hbox{$\lambda\geq1/2$.}
            \end{array}
            \right.
\end{align}
If $-1/2<\lambda<0$, then for all $n\in\mathbb{N}$,
\begin{align}\label{eq:BoundII}
\max_{|x|\leq1}|C_{n}^{\lambda}(x)| \leq n^{\lambda-1} \left\{
\begin{array}{ll}
{\displaystyle 2^{1-\lambda} |\lambda|},               & \hbox{$n=2,4,6,\ldots$,}   \\[8pt]
{\displaystyle \frac{2|\lambda|}{\sqrt{1+2\lambda}} }, &
\hbox{$n=1,3,5,\ldots$.}
            \end{array}
            \right.
\end{align}
\end{lemma}
\begin{proof}
As for \eqref{eq:BoundI}, it follows by combining the inequality
$|C_{n}^{\lambda}(x)|\leq C_{n}^{\lambda}(1)=(2\lambda)_n/n!$ with
Lemma \ref{lem:gamma}. As for \eqref{eq:BoundII}, it follows by
combining equations (18.14.5) and (18.14.6) in Olver \emph{et al.}
(2010) with Lemma \ref{lem:gamma}.
\end{proof}

Finally, we note that Gegenbauer polynomials include some important
polynomials such as Legendre and Chebyshev polynomials as special
cases, and more specifically,
\begin{align}\label{eq:GegChebU}
P_n(x) = C_n^{1/2}(x),  \quad  U_n(x) = C_n^{1}(x), \quad n\geq 0,
\end{align}
where $P_n(x)$ is the Legendre polynomial of degree $n$ and $U_n(x)$
is the Chebyshev polynomial of the second kind of degree $n$. When
$\lambda = 0$, the Gegenbauer polynomials reduce to the Chebyshev
polynomials of the first kind by the following definition
\begin{align}\label{eq:GegChebT}
\lim_{\lambda\rightarrow0^{+}} \lambda^{-1} C_{n}^{\lambda}(x) =
\frac{2}{n} T_n(x), \quad n \geq 1,
\end{align}
where $T_n(x)$ is the Chebyshev polynomial of the first kind of
degree $n$.

\section{Experimental observations}\label{sec:experiment}
In this section we carry out some numerical experiments to compare
the convergence behavior of $\mathcal{B}_n(f)$ and
$S_n^{\lambda}(f)$. In order to quantify the discrepancy between the
rates of convergence of both methods, we introduce the quantity
\begin{align}\label{eq:Indicator}
\mathcal{R}^{\lambda}(n) = \frac{\|f - S_n^{\lambda}(f) \|_{\infty}
}{\|f - \mathcal{B}_n(f) \|_{\infty}} \geq1.
\end{align}
Moreover, using \eqref{eq:GegenOrthog}, the Gegenbauer projection
$S_n^{\lambda}(f)$ can be written as
\begin{align}\label{eq:GegenSeries}
S_n^{\lambda}(f) &= \sum_{k=0}^{n} a_k^{\lambda} C_k^{\lambda}(x),
\quad a_k^{\lambda} = \frac{1}{h_k^{\lambda} } \int_{\Omega}
\omega_{\lambda}(x) C_k^{\lambda}(x) f(x) \mathrm{d}x.
\end{align}
In our computations, we compute $\mathcal{B}_n(f)$ using the
barycentric-Remez algorithm (Pachon \& Trefethen, 2009) and its
implementation is available in Chebfun with the \texttt{minimax}
command (Driscoll \emph{et al.}, 2014). Moreover, the maximum error
of $S_n^{\lambda}(f)$ is measured by using a finer grid in $\Omega$.
Throughout the rest of the paper, we may use $S_n^{\lambda}(f,x)$
instead of $S_n^{\lambda}(f)$ when computing $S_n^{\lambda}(f)$ at
the point $x$.

\begin{figure}[t!]
\centering\includegraphics[trim=36mm 0mm 0mm
0mm,width=16cm,height=9cm]{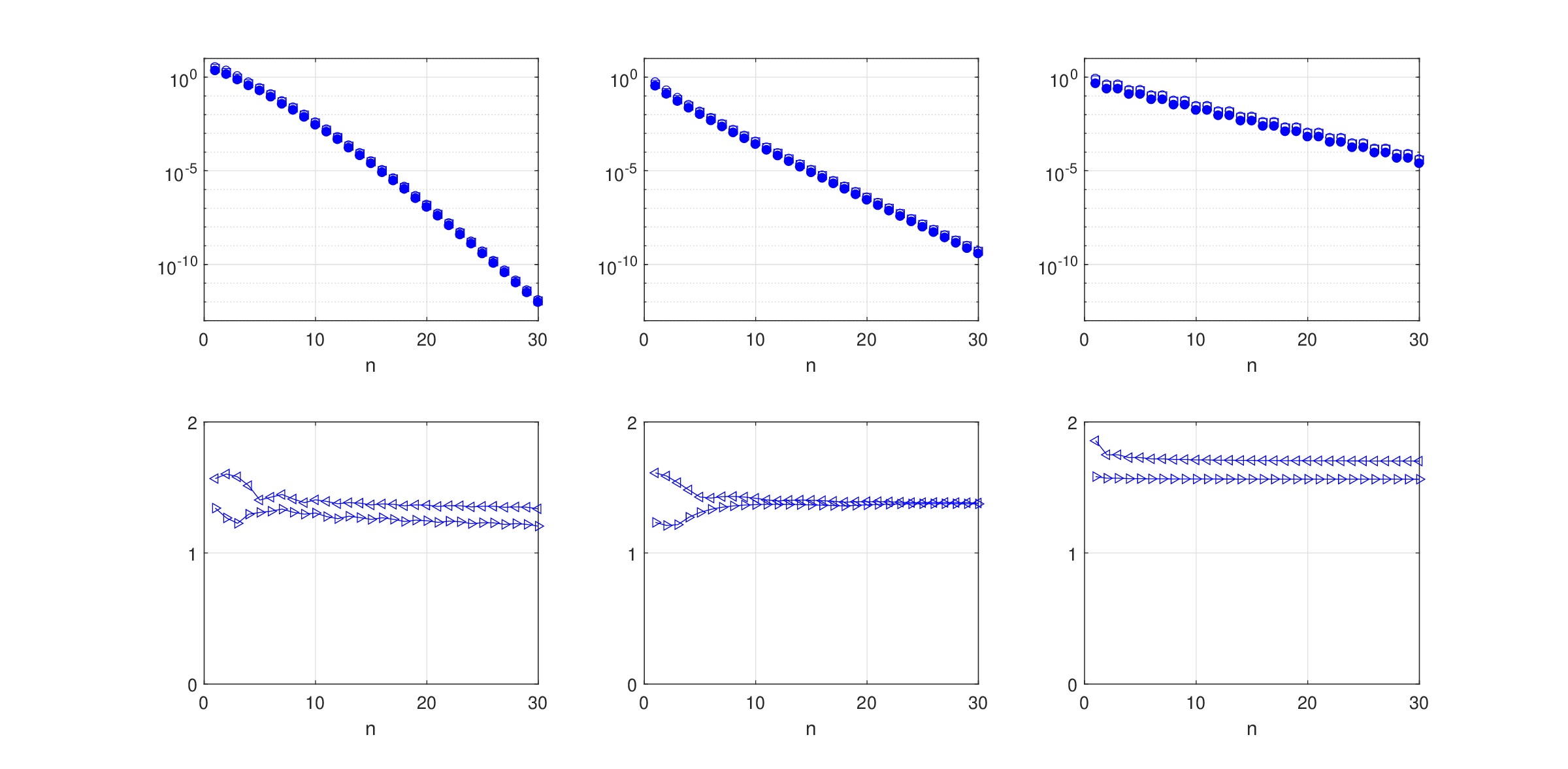} \caption{Top row shows the
log plot of the maximum errors of $\mathcal{B}_n(f)$ ($\bullet$) and
$S_n^{\lambda}(f)$ with $\lambda=-2/5$ ($\circ$) and $\lambda=-1/10$
($\Box$), for $f_1$ (left), $f_2$ (middle) and $f_3$ (right). Bottom
row shows the plot of the corresponding $\mathcal{R}^{\lambda}(n)$
for $\lambda=-2/5$ ($\triangleleft$) and $\lambda=-1/10$
($\triangleright$).} \label{fig:ExamI}
\end{figure}

\subsection{Analytic functions}
We consider the following three test functions
\begin{align}\label{eq:TestFun1}
f_1(x) = e^{2x^3}, \quad f_2(x) = \ln(1.2+x), \quad  f_3(x) =
1/(1+9x^2).
\end{align}
We divide the choice of the parameter $\lambda$ into two ranges:
$\lambda\in(-1/2,0]$ and $\lambda>0$. Figure \ref{fig:ExamI}
illustrates the maximum errors of $\mathcal{B}_n(f)$ and
$S_n^{\lambda}(f)$ for $\lambda=-2/5$ and $\lambda=-1/10$ and the
quantity $\mathcal{R}^{\lambda}(n)$ as a function of $n$. From the
top row of Figure \ref{fig:ExamI}, we see that the maximum error of
$\mathcal{B}_n(f)$ is indistinguishable with that of
$S_n^{\lambda}(f)$. From the bottom row of Figure \ref{fig:ExamI},
we see that these two $\mathcal{R}^{\lambda}(n)$ tend, respectively,
to some finite constants as $n$ grows, and thus the rate of
convergence of $S_n^{\lambda}(f)$ is the same as that of
$\mathcal{B}_n(f)$. Figure \ref{fig:ExamII} illustrates the maximum
errors of $\mathcal{B}_n(f)$ and $S_n^{\lambda}(f)$ for $\lambda=1$
and $\lambda=2$ and $n^{-\lambda}\mathcal{R}^{\lambda}(n)$ as a
function of $n$. From the top row of Figure \ref{fig:ExamII}, we see
clearly that the rate of convergence of $\mathcal{B}_n(f)$ is faster
than that of $S_n^{\lambda}(f)$. From the bottom row of Figure
\ref{fig:ExamII}, we see that these two
$n^{-\lambda}\mathcal{R}^{\lambda}(n)$ tend, respectively, to some
finite constants as $n$ grows, which imply that the rate of
convergence of $S_n^{\lambda}(f)$ is slower than that of
$\mathcal{B}_n(f)$ by a factor of $n^{\lambda}$.

\begin{figure}[ht]
\centering
\includegraphics[trim=36mm 0mm 0mm 0mm,width=16cm,height=9cm]{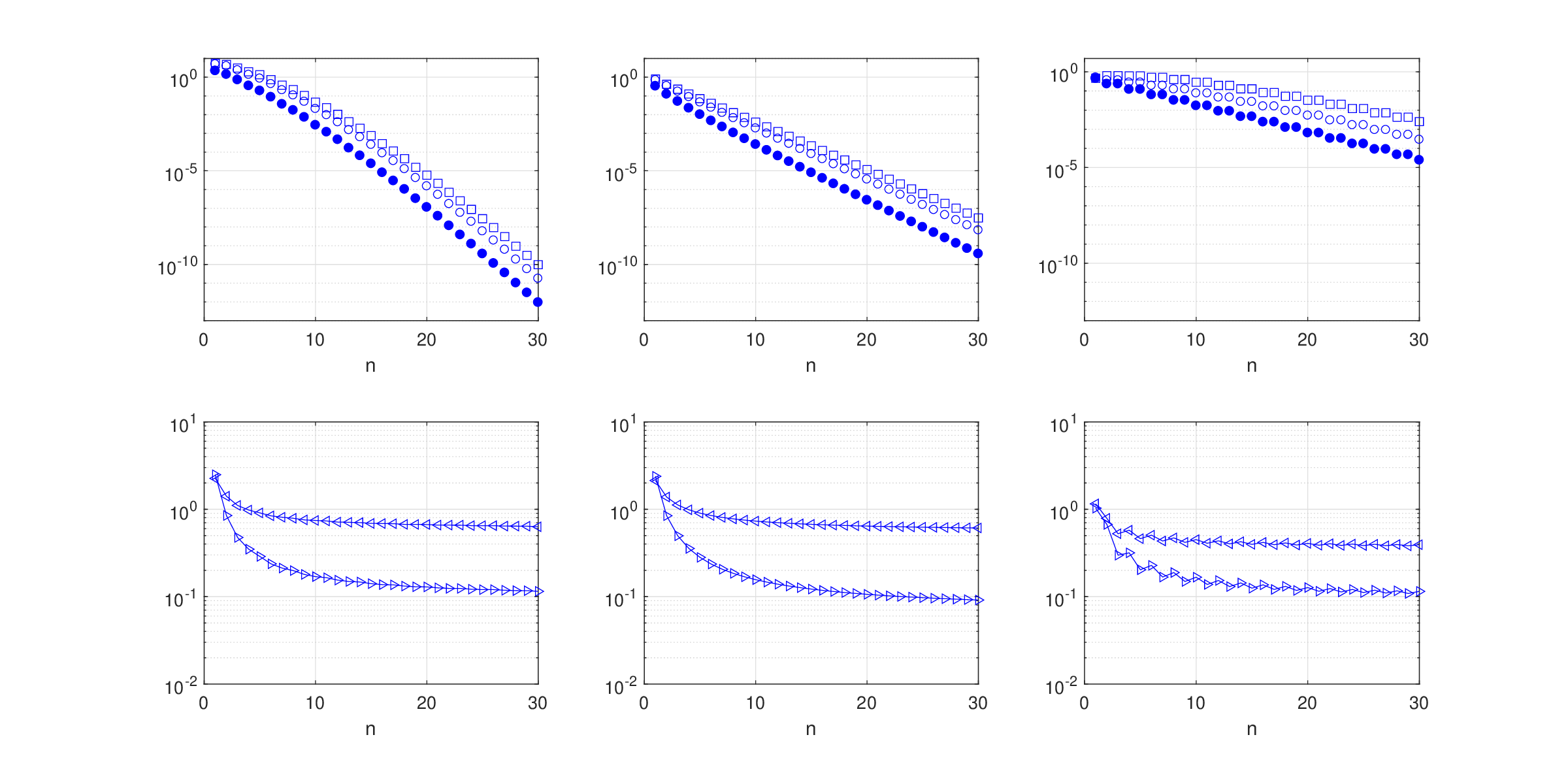}
\caption{Top row shows the log plot of the maximum errors of
$\mathcal{B}_n(f)$ ($\bullet$) and $S_n^{\lambda}(f)$ with
$\lambda=1$ ($\circ$) and $\lambda=2$ ($\Box$), for $f_1$ (left),
$f_2$ (middle) and $f_3$ (right). Bottom row shows the log plot of
the corresponding $n^{-\lambda}\mathcal{R}^{\lambda}(n)$ for
$\lambda=1$ ($\triangleleft$) and $\lambda=2$ ($\triangleright$). }
\label{fig:ExamII}
\end{figure}

In summary, the above observations suggest the following
conclusions:
\begin{itemize}
\item For $\lambda\in(-1/2,0]$, the rate of convergence of
$S_n^{\lambda}(f)$ is the same as that of $\mathcal{B}_n(f)$;

\item For $\lambda>0$, however, the rate of convergence of
$S_n^{\lambda}(f)$ is slower than that of $\mathcal{B}_n(f)$ by a
factor of $n^{\lambda}$.
\end{itemize}

\begin{figure}[ht]
\centering
\includegraphics[trim = 36mm 0mm 0mm 0mm,width=16cm,height=9cm]{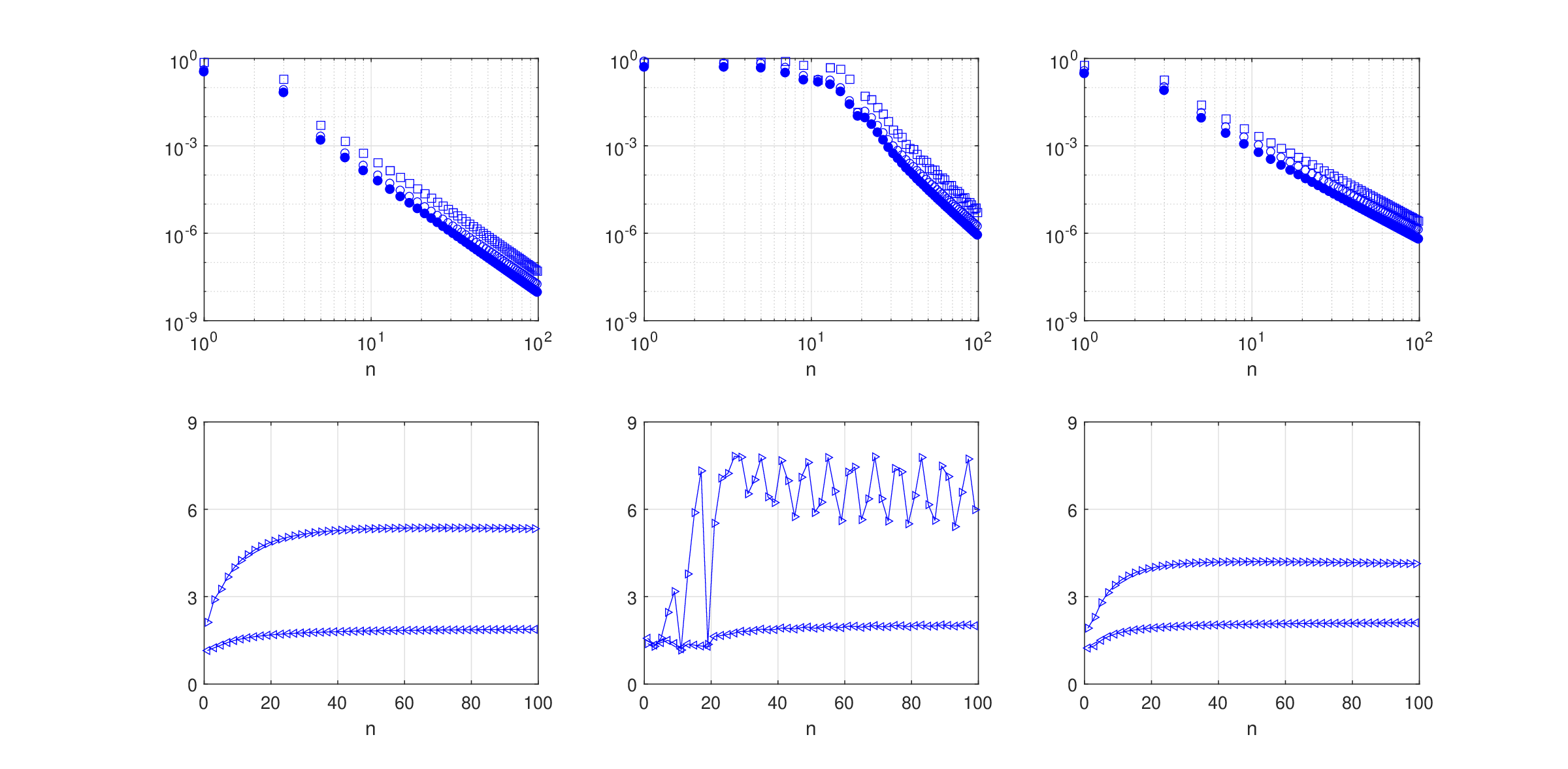}
\caption{Top row shows the log-log plot of the maximum errors of
$\mathcal{B}_n(f)$ ($\bullet$), $S_n^{\lambda}(f)$ with
$\lambda=-1/5$ ($\circ$) and $\lambda=9/10$ ($\Box$), for $f_4$
(left), $f_5$ (middle) and $f_6$ (right). Bottom row shows the plot
of the corresponding $\mathcal{R}^{\lambda}(n)$ for $\lambda=-1/5$
($\triangleleft$) and $\lambda=9/10$ ($\triangleright$). }
\label{fig:ExamIII}
\end{figure}

\subsection{Differentiable functions}
We consider the following test functions
\begin{align}\label{eq:TestFun2}
f_4(x) = (x)_{+}^4, \quad f_5(x) = \left| \sin(4x) \right|^5, \quad
f_6(x) = \left\{
\begin{array}{ll}
{\displaystyle 2\cos(x) },        & \hbox{$x<0$,}   \\[8pt]
{\displaystyle 2x^3 - x^2 + 2 },   & \hbox{$x\geq0$,}
            \end{array}
            \right.
\end{align}
where $(x)_{+}^k$ is the {\it truncated power function} defined by
\begin{align}
(x)_{+}^k = \left\{
\begin{array}{ll}
{\displaystyle x^k }, & \hbox{$x\geq0$,}   \\[8pt]
{\displaystyle 0 },   & \hbox{$x<0$,}
            \end{array}
            \right. ~~ k\geq1,~~ \mbox{and} ~~
(x)_{+}^0 = \left\{
\begin{array}{ll}
{\displaystyle 1 }, & \hbox{$x\geq0$,}   \\[8pt]
{\displaystyle 0 }, & \hbox{$x<0$.}
            \end{array}
            \right.
\end{align}
As will become clear later, the above three functions belong to the
space of piecewise analytic functions of class ${C}^{m-1}(\Omega)$
with $m=4,5,3$, respectively. In our numerical tests, we divide the
choice of the parameter $\lambda$ into ranges: $\lambda \in
(-1/2,1]$ and $\lambda>1$. Figure \ref{fig:ExamIII} illustrates the
maximum errors of $\mathcal{B}_n(f)$ and $S_n^{\lambda}(f)$ for
$\lambda=-1/5$ and $\lambda=9/10$ and the quantity
$\mathcal{R}^{\lambda}(n)$ as a function of $n$. From the top row of
Figure \ref{fig:ExamIII}, we see that the maximum error of
$S_n^{\lambda}(f)$ is slightly worse than that of
$\mathcal{B}_n(f)$. From the bottom row of Figure \ref{fig:ExamIII},
we see that these two $\mathcal{R}^{\lambda}(n)$ tend to or
oscillate around some finite constants as $n$ grows, which imply
that the rate of convergence of $S_n^{\lambda}(f)$ is the same as
that of $\mathcal{B}_n(f)$. Figure \ref{fig:ExamIV} illustrates the
maximum errors of $\mathcal{B}_n(f)$ and $S_n^{\lambda}(f)$ for
$\lambda=3/2$ and $\lambda=3$ and
$n^{1-\lambda}\mathcal{R}^{\lambda}(n)$ as a function of $n$. From
the top row of Figure \ref{fig:ExamIV}, we see that the rate of
convergence of $S_n^{\lambda}(f)$ is obviously slower than that of
$\mathcal{B}_n(f)$. From the bottom row of Figure \ref{fig:ExamIV},
we see that these two $n^{1-\lambda}\mathcal{R}^{\lambda}(n)$ tend
to or oscillate around some finite constants as $n$ grows, which
imply that the rate of convergence of $S_n^{\lambda}(f)$ is slower
than that of $\mathcal{B}_n(f)$ by a factor of $n^{\lambda-1}$.

\begin{figure}[ht]
\centering
\includegraphics[trim = 36mm 0mm 0mm 0mm,width=16cm,height=9cm]{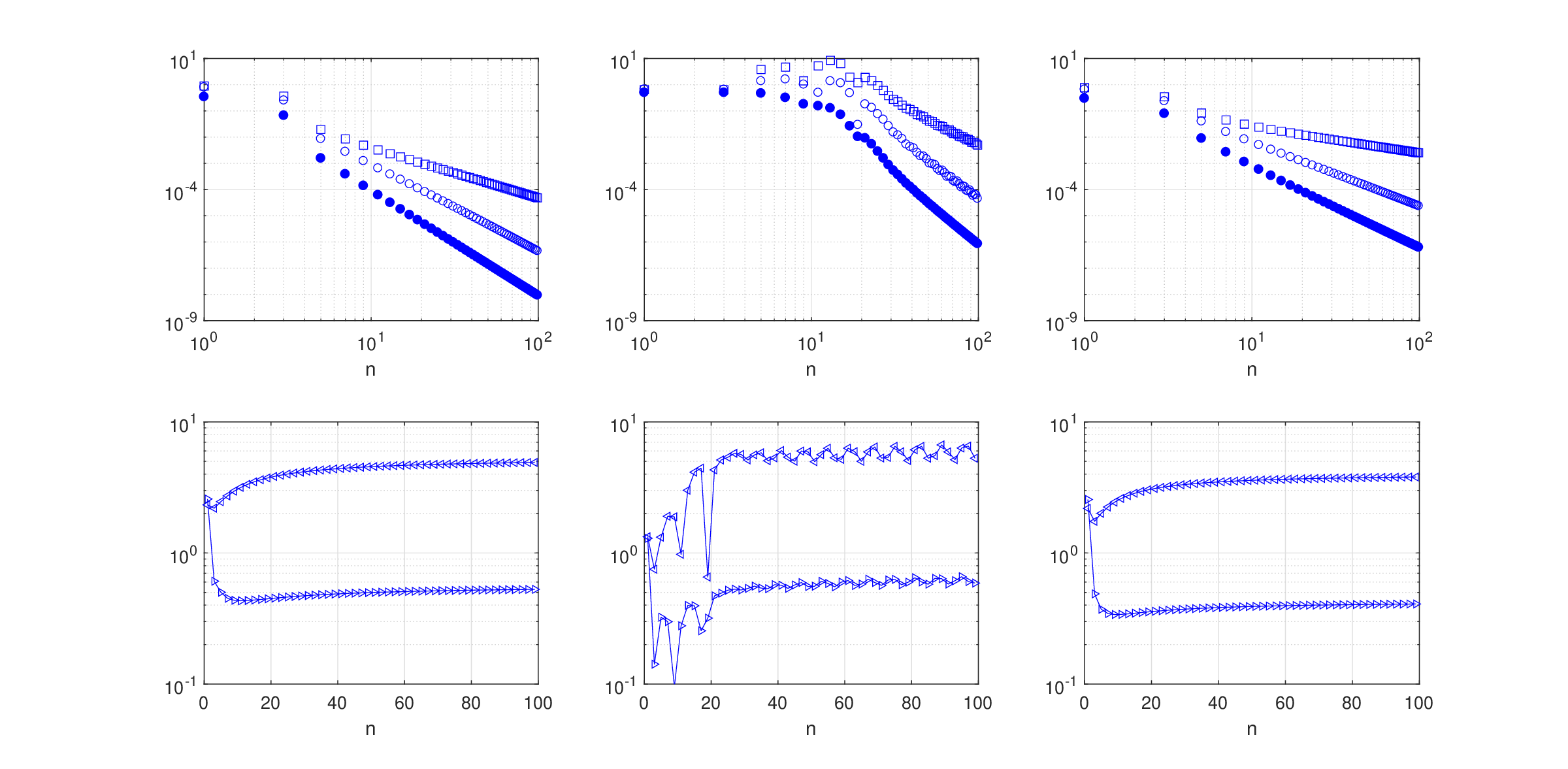}
\caption{Top row shows the log-log plot of the maximum errors of
$\mathcal{B}_n(f)$ ($\bullet$), $S_n^{\lambda}(f)$ with
$\lambda=3/2$ ($\circ$) and $\lambda=3$ ($\Box$), for $f_4$ (left),
$f_5$ (middle) and $f_6$ (right). Bottom row shows the log plot of
the corresponding $n^{1-\lambda}\mathcal{R}^{\lambda}(n)$ for
$\lambda=3/2$ ($\triangleleft$) and $\lambda=3$ ($\triangleright$).}
\label{fig:ExamIV}
\end{figure}

In summary, the above observations suggest the following
conclusions:
\begin{itemize}
\item For $\lambda\in(-1/2,1]$, the rate of convergence of
$S_n^{\lambda}(f)$ is the same as that of $\mathcal{B}_n(f)$;

\item For $\lambda>1$, however, the rate of convergence of
$S_n^{\lambda}(f)$ is slower than that of $\mathcal{B}_n(f)$ by a
factor of $n^{\lambda-1}$, which is one power of $n$ smaller than
the predicted result using \eqref{eq:errorG} and
\eqref{eq:Lebesgue}.
\end{itemize}
In the following sections, we shall carry out a convergence rate
analysis of $S_n^{\lambda}(f)$ to explain these observations. We
remark that the convergence results of the particular case
$\lambda=0$ (that corresponds to Chebyshev projections) have been
included in the above two observations. We refer to (Liu \emph{et
al.}, 2019; Trefethen, 2013) for more details on the convergence
rate analysis of Chebyshev projections and to Wang (2021) for a
comparison of Chebyshev, Legendre projections and
$\mathcal{B}_n(f)$. Hereafter, we will omit discussion of this case.

\section{Explicit and optimal error bounds of Gegenbauer projections for analytic functions}
\label{sec:Analytic} In this section, we establish some new error
bounds of Gegenbauer projections for analytic functions. Let
$\mathcal{E}_{\rho}$ denote the Bernstein ellipse
\begin{equation}\label{def:Bern}
\mathcal{E}_{\rho} = \left\{ z \in \mathbb{C} ~\bigg|~ z = \frac{u +
u^{-1}}{2},~~ |u| = \rho\geq1 \right\},
\end{equation}
and it has foci at $\pm 1$ and the major and minor semi-axes are
given by $(\rho+\rho^{-1})/2$ and $(\rho-\rho^{-1})/2$,
respectively.

The starting point of our analysis is the contour integral
expression of the Gegenbauer coefficients, which was derived in
Cantero \& Iserles (2012) by rearranging the Taylor expansion and in
Wang (2016) by rearranging the Chebyshev expansion. Here, we propose
an alternative way for deriving the contour integral expression
using Cauchy's integral formula and a connection formula between the
associated Legendre functions of the second kind and hypergeometric
functions.
\begin{lemma}\label{lem:Contour}
Suppose that $f$ is analytic in the region bounded by the ellipse
$\mathcal{E}_{\rho}$ for some $\rho>1$, then for each $k\geq0$ and
$\lambda>-1/2$ and $\lambda\neq0$,
\begin{align}\label{eq:Contour}
a_k^{\lambda} & = \frac{c_{k,\lambda} }{i\pi}
\oint_{\mathcal{E}_{\rho}} \frac{f(z)}{ (z\pm\sqrt{z^2 - 1})^{k+1} }
{}_2\mathrm{ F}_1\left[\begin{matrix} k + 1, & 1 - \lambda
\\   k + \lambda + 1 \hspace{-1cm} &\end{matrix} ; \frac{1}{(z\pm\sqrt{z^2 - 1}
)^{2}} \right]  \mathrm{d}z,
\end{align}
where $i$ is the imaginary unit and the sign in $z \pm \sqrt{z^2 -
1}$ is chosen so that $|z\pm\sqrt{z^2 - 1}|>1$ and
\begin{align}\label{eq:c}
c_{k,\lambda} = \frac{\Gamma(\lambda)
\Gamma(k+1)}{\Gamma(k+\lambda)}.
\end{align}
\end{lemma}
\begin{proof}
By Cauchy's integral formula and exchanging the order of
integration, we obtain
\begin{align}\label{eq:ContourStep1}
a_k^{\lambda} &= \frac{1}{h_k^{\lambda}} \int_{\Omega}
\omega_{\lambda}(x) C_k^{\lambda}(x) \left(\frac{1}{2\pi i}
\oint_{\mathcal{E}_{\rho}} \frac{f(z)}{z-x} \mathrm{d}z \right)
\mathrm{d}x \nonumber \\
&= \frac{1}{\pi i} \oint_{\mathcal{E}_{\rho}} f(z) \left( \frac{1}{2
h_k^{\lambda}} \int_{\Omega} \frac{\omega_{\lambda}(x)
C_k^{\lambda}(x)}{z-x} \mathrm{d}x \right) \mathrm{d}z.
\end{align}
We denote by $\Upsilon$ the term inside the bracket in the last
equality. From (Gradshteyn \& Ryzhik, 2007, Equation (7.312.1)) we
know that $\Upsilon$ can be expressed in the form
\begin{align}
\Upsilon &= \frac{\pi^{1/2}
2^{3/2-\lambda}}{2\Gamma(\lambda)h_k^{\lambda}} e^{-(\lambda-1/2)\pi
i} (z^2-1)^{\lambda/2-1/4} Q_{k+\lambda-1/2}^{\lambda-1/2}(z),
\nonumber
\end{align}
where $Q_{\nu}^{\mu}(z)$ is the associated Legendre function of the
second kind of degree $\nu$ and order $\mu$. Furthermore, using the
connection formula between $Q_{\nu}^{\mu}(z)$ and ${}_2\mathrm{
F}_1(\cdot)$ in (Gradshteyn \& Ryzhik, 2007, Equation (8.777.2)) and
the last transformation formula of ${}_2\mathrm{ F}_1(\cdot)$ in
(Gradshteyn \& Ryzhik, 2007, Equation (9.131.1)), we have that
\begin{align}
\Upsilon &= \frac{c_{k,\lambda}}{2^{1-2\lambda}}
\frac{(z^2-1)^{\lambda-1/2}}{(z\pm\sqrt{z^2-1})^{k+2\lambda}}
{}_2\mathrm{ F}_1\left[\begin{matrix} k + 2\lambda, &  \lambda
\\   k + \lambda + 1 \hspace{-1cm} &\end{matrix}~ ; \frac{1}{(z\pm\sqrt{z^2-1})^{2}}
\right], \nonumber \\
&= \frac{c_{k,\lambda}}{(z\pm\sqrt{z^2-1})^{k+1}} {}_2\mathrm{
F}_1\left[\begin{matrix} k + 1, & 1-\lambda
\\   k + \lambda + 1 \hspace{-1cm} &\end{matrix}~ ; \frac{1}{(z\pm\sqrt{z^2-1})^{2}}
\right]. \nonumber
\end{align}
Substituting this into \eqref{eq:ContourStep1} gives the desired
result. This completes the proof.
\end{proof}

We now state some new bounds on the Gegenbauer coefficients
$\{a_k^{\lambda}\}$ for all $\lambda>-1/2$ and $\lambda\neq0$.
Compared to the previous results in Wang (2016), our bounds are new
whenever $-1/2<\lambda<0$ and are more concise whenever $\lambda>0$.
\begin{theorem}\label{thm:GegBound}
Under the assumptions of Lemma \ref{lem:Contour}, we have for
$\lambda\neq0$ that
\begin{align}\label{eq:GenBound}
|a_0^{\lambda}| \leq D(\lambda,\rho) \left\{
            \begin{array}{ll}
{\displaystyle \frac{1}{|\Gamma(\lambda)|} }, & \hbox{$-1/2<\lambda<0$,}   \\[12pt]
{\displaystyle \lambda }, & \hbox{$0<\lambda\leq 1$,}   \\[5pt]
{\displaystyle \frac{1}{\Gamma(\lambda)} }, & \hbox{$\lambda>1$,}
            \end{array}
            \right.  \quad  |a_k^{\lambda}| &\leq D(\lambda,\rho)
            \frac{k^{1-\lambda}}{\rho^{k}}, \quad k\geq1,
\end{align}
where $D(\lambda,\rho)$ is defined by
\begin{align}\label{eq:D}
D(\lambda,\rho) &= \frac{\displaystyle M
L(\mathcal{E}_{\rho})}{\pi\rho} \left\{   \begin{array}{ll}
{\displaystyle
\frac{\Gamma(1+\lambda)^2\Gamma(1-2\lambda)}{(-\lambda)\Gamma(1-\lambda)}
\left(1-\frac{1}{\rho^2} \right)^{2\lambda-1}}, & \hbox{$-1/2<\lambda<0$,}   \\[15pt]
{\displaystyle \frac{1}{\lambda} \left(1 - \frac{1}{\rho^2} \right)^{\lambda-1}}, & \hbox{$0<\lambda\leq1$,}   \\[15pt]
{\displaystyle \Gamma(\lambda) \left(1 + \frac{1}{\rho^2}
\right)^{\lambda-1}}, & \hbox{$\lambda>1$,}
            \end{array}
            \right.
\end{align}
and $M=\max_{z\in \mathcal{E}_{\rho}}|f(z)|$ and
$L(\mathcal{E}_{\rho})$ is the length of the circumference of
$\mathcal{E}_{\rho}$.
\end{theorem}
\begin{proof}
We follow the same line as that in Wang (2016). From Lemma
\ref{lem:Contour} and (Wang, 2016, Theorem~4.1) we have that
\begin{align}\label{eq:GenS1}
|a_k^{\lambda}| &\leq \frac{|c_{k,\lambda}| M
L(\mathcal{E}_{\rho})}{\pi \rho^{k+1}} \left\{
            \begin{array}{ll}
{\displaystyle {}_2\mathrm{ F}_1\left[\begin{matrix} k + 1, &
1-\lambda \\  k+\lambda+1 \hspace{-1cm} &\end{matrix} ; ~
\frac{1}{\rho^2}
\right]}, & \hbox{$-1/2<\lambda\leq 1$ and $\lambda\neq0$,}   \\[15pt]
{\displaystyle  {}_2\mathrm{ F}_1\left[\begin{matrix} k + 1, &
1-\lambda
\\  k+\lambda+1 \hspace{-1cm} &\end{matrix} ;~ -\frac{1}{\rho^2}
\right]}, & \hbox{$\lambda>1$.}
            \end{array}
            \right.
\end{align}
It remains to bound $c_{k,\lambda}$ and these hypergeometric
functions on the right-hand side of \eqref{eq:GenS1}. For the
former, it is easily seen that $|c_{k,\lambda}|=1$ when $k=0$. For
$k\geq1$, using Lemma \ref{lem:gamma} we obtain
\begin{align}
|c_{k,\lambda}| \leq k^{1-\lambda} \left\{
\begin{array}{ll}
{\displaystyle |\Gamma(\lambda)| }, & \hbox{$-1/2 <\lambda<0$,}   \\[6pt]
{\displaystyle \lambda^{-1} },      & \hbox{$0<\lambda \leq 1$,}  \\[6pt]
{\displaystyle \Gamma(\lambda) },   & \hbox{$\lambda>1$.}
            \end{array}
            \right.
\end{align}
Next, we consider the bound of these hypergeometric functions on the
right-hand side of \eqref{eq:GenS1}. For $\lambda>0$ and $|z|<1$,
using the Euler integral representation of the Gauss hypergeometric
function (Olver \emph{et al.}, 2010, Equation~(15.6.1)), we obtain
\begin{align}\label{eq:GenS2}
\left| {}_2\mathrm{ F}_1\left[\begin{matrix} k + 1, & 1-\lambda
\\   k+\lambda+1  \hspace{-1cm} &\end{matrix} ;~ z
\right] \right| &=
\frac{\Gamma(k+\lambda+1)}{\Gamma(k+1)\Gamma(\lambda)}
\left| \int_{0}^{1} t^{k} (1-t)^{\lambda-1} \left(1 - zt\right)^{\lambda-1} \mathrm{d}t \right| \nonumber \\[8pt]
&\leq \left\{
            \begin{array}{ll}
{\displaystyle \left(1 - |z| \right)^{\lambda-1}}, & \hbox{$0<\lambda\leq 1$,}   \\[8pt]
{\displaystyle \left(1 + |z| \right)^{\lambda-1}}, &
\hbox{$\lambda>1$.}
            \end{array}
            \right.
\end{align}
For $-1/2<\lambda<0$, it is easily verified that
\begin{align}
\frac{(k+1)_j}{(k+\lambda+1)_j} \leq \frac{(1)_j}{(\lambda+1)_j},
\quad   \frac{(1-\lambda)_j}{(1+\lambda)_j} \leq
\frac{\Gamma(1+\lambda)\Gamma(1-2\lambda)}{\Gamma(1-\lambda)}
\frac{(1-2\lambda)_j}{(1)_j}, \nonumber
\end{align}
and therefore
\begin{align}\label{eq:GenS3}
\left| {}_2\mathrm{ F}_1\left[\begin{matrix} k + 1, & 1-\lambda
\\   k+\lambda+1  \hspace{-1cm} &\end{matrix}; ~ z
\right] \right| &\leq \sum_{j=0}^{\infty} \frac{(k+1)_j
(1-\lambda)_j}{(k+\lambda+1)_j} \frac{|z|^{j}}{j!} \leq
\sum_{j=0}^{\infty} \frac{(1)_j (1-\lambda)_j}{(\lambda+1)_j}
\frac{|z|^{j}}{j!} \nonumber \\
&\leq \frac{\Gamma(1+\lambda)\Gamma(1-2\lambda)}{\Gamma(1-\lambda)}
\sum_{j=0}^{\infty} \frac{(1-2\lambda)_j}{j!} |z|^{j} \nonumber \\
&= \frac{\Gamma(1+\lambda)\Gamma(1-2\lambda)}{\Gamma(1-\lambda)} (1
- |z|)^{2\lambda-1}.
\end{align}
Combining \eqref{eq:GenS1}, \eqref{eq:GenS2} and \eqref{eq:GenS3},
the desired bounds follow immediately.
\end{proof}

With the above result, we are now ready to establish error bounds
for Gegenbauer projections in the maximum norm, and these bounds are
fully explicit with respect to the parameters $\lambda$, $\rho$ and
$n$ and are more informative than existing results. Throughout the
paper, $\lfloor x \rfloor$ denotes the integer part of $x$.
\begin{theorem}\label{thm:RateAnal}
Suppose that $f$ is analytic in the region bounded by the ellipse
$\mathcal{E}_{\rho}$ for some $\rho>1$, and let $D(\lambda,\rho)$ be
defined by \eqref{eq:D}.
\begin{enumerate}
\item[\rm (i)] If $\lambda>0$, then for
$n>\lfloor\eta\lambda/((\eta-1)\ln\rho) \rfloor$ and $\eta>1$ is
arbitrary,
\begin{align}\label{eq:AnalBoundI}
\|f - S_n^{\lambda}(f) \|_{\infty} \leq \mathcal{K}
\frac{n^{\lambda}}{\rho^n},
\end{align}
where $\mathcal{K}$ is defined by
\begin{align}
\mathcal{K}= \eta \frac{D(\lambda,\rho)}{\ln\rho} \left\{
            \begin{array}{ll}
{\displaystyle \frac{1}{\Gamma(2\lambda)} }, & \hbox{$0<\lambda\leq 1/2$,}   \\[10pt]
{\displaystyle  2\lambda}, & \hbox{$\lambda>1/2$.}
            \end{array}
            \right. \nonumber
\end{align}

\item[\rm (ii)] If $-1/2<\lambda<0$, then for $n\geq0$,
\begin{align}\label{eq:AnalBoundII}
\| f - S_n^{\lambda}(f) \|_{\infty} \leq \frac{2|\lambda|
D(\lambda,\rho)}{\sqrt{1+2\lambda}(\rho-1) \rho^n}.
\end{align}
\end{enumerate}
Moreover, up to constant factors, these bounds on the right-hand
side of \eqref{eq:AnalBoundI} and \eqref{eq:AnalBoundII} are optimal
in the sense that they can not be improved in any negative powers of
$n$ further.
\end{theorem}
\begin{proof}
For part (i), combining Lemma \ref{lem:Bound} with Theorem
\ref{thm:GegBound} gives
\begin{align}
\|f - S_n^{\lambda}(f) \|_{\infty} &\leq D(\lambda,\rho) \left(
\sum_{k=n+1}^{\infty} \frac{k^{\lambda}}{\rho^{k}} \right) \left\{
            \begin{array}{ll}
{\displaystyle \frac{1}{\Gamma(2\lambda)}}, & \hbox{$0<\lambda\leq 1/2$,}   \\[10pt]
{\displaystyle 2\lambda}, & \hbox{$\lambda>1/2$.}
            \end{array}
            \right. \nonumber
\end{align}
For the sum inside the bracket, one can easily check that
$k^{\lambda}/\rho^k$ is strictly decreasing with respect to $k$
whenever $k\geq\lambda/\ln\rho$, and thus
\begin{align}\label{eq:BoundS2}
\sum_{k=n+1}^{\infty} \frac{k^{\lambda}}{\rho^{k}} &\leq
\int_{n}^{\infty} \frac{x^{\lambda}}{\rho^{x}} \mathrm{d}x =
\frac{\Gamma(\lambda+1,n\ln\rho)}{(\ln\rho)^{1+\lambda}},
\end{align}
where $\Gamma(a,x)$ is the incomplete gamma function (see, e.g.,
Olver \emph{et al.}, p.~174). Furthermore, from Natalini \& Palumbo
(2000) we know that $|\Gamma(a,x)|\leq \eta x^{a-1} e^{-x}$ for
$a>1$ and $x>(a-1)\eta/(\eta-1)$ and $\eta>1$ is arbitrary, the
desired result \eqref{eq:AnalBoundI} follows. The proof of part (ii)
is similar and we omit the details.

We now turn to prove the optimality of \eqref{eq:AnalBoundI} and
\eqref{eq:AnalBoundII}. Here we only prove the former since the
latter can be proved by a similar argument. Suppose by contradiction
that there exist constants $\gamma,c>0$ independent of $n$ such that
\begin{align}\label{eq:LegProjBoundS3}
\|f-S_n^{\lambda}(f) \|_{\infty} \leq c
\frac{n^{\lambda-\gamma}}{\rho^n}.
\end{align}
We consider the function $f(x)=1/(x-\omega)$ with
$\omega+\sqrt{\omega^2-1}>1+\lambda^{-1}$. It is easily seen that
this function has a simple pole at $x=\omega$ and therefore
$\rho\leq\omega+\sqrt{\omega^2-1}-\epsilon$, where $\epsilon>0$ may
be taken arbitrary small. Using Lemma \ref{lem:Contour} and the
residue theorem, we can write the Gegenbauer coefficients of $f(x)$
as
\begin{align}\label{eq:PoleCoeff}
a_k^{\lambda} & = \frac{(-2
c_{k,\lambda})}{(\omega+\sqrt{\omega^2-1})^{k+1}} {}_2\mathrm{
F}_1\left[\begin{matrix} k + 1, & 1-\lambda
\\   k + \lambda + 1 \hspace{-1cm} &\end{matrix} ;~ \frac{1}{(\omega+\sqrt{\omega^2-1})^{2}}
\right].
\end{align}
Clearly, we see that $a_k^{\lambda}<0$ for all $k\geq0$. Moreover,
by considering the ratio $a_{k+1}^{\lambda}/a_k^{\lambda}$, it is
not difficult to verify that the sequence
$\{a_{k}^{\lambda}\}_{k=0}^{\infty}$ is strictly increasing. We now
consider the error of $S_n^{\lambda}(f)$ at the point $x=1$. Recall
the well-known inequality $\max_{|x|\leq1}|C_{k}^{\lambda}(x)|\leq
C_{k}^{\lambda}(1)$ for $\lambda>0$ and $k\in \mathbb{N}$, we obtain
that
\begin{align}
| f(1) - S_n^{\lambda}(f,1) | &= -\sum_{k=n+1}^{\infty}
a_{k}^{\lambda} C_{k}^{\lambda}(1) \geq -a_{n+1}^{\lambda}
C_{n+1}^{\lambda}(1). \nonumber
\end{align}
Combining this with \eqref{eq:LegProjBoundS3} we deduce that
\begin{align}
-a_{n+1}^{\lambda} C_{n+1}^{\lambda}(1) \leq \|f(x) -
S_n^{\lambda}(f) \|_{\infty} \leq c
\frac{n^{\lambda-\gamma}}{\rho^n}.
\end{align}
By using \eqref{eq:GenS2}, \eqref{eq:AsyGAMMA} and
\eqref{eq:PoleCoeff}, we obtain that $|a_{n+1}^{\lambda}
C_{n+1}^{\lambda}(1)|=O(n^{\lambda}
(\omega+\sqrt{\omega^2-1})^{-n})$. On the other hand, we know that
$n^{\lambda-\gamma}\rho^{-n}=O(n^{\lambda-\gamma}
(\omega+\sqrt{\omega^2-1}-\epsilon)^{-n})$. This leads to a
contradiction since the upper bound may be smaller than the lower
bound when $\epsilon$ is sufficiently small. Therefore, we can
conclude that the derived bound \eqref{eq:AnalBoundI} is optimal and
can not be improved in any negative powers of $n$. This completes
the proof.
\end{proof}

\begin{remark}
From Cheney (1998) and Bernstein (1912) we know that $\|f -
\mathcal{B}_n(f) \|_{\infty}=O(\rho^{-n})$. Comparing this with
\eqref{eq:AnalBoundI} and \eqref{eq:AnalBoundII}, it is easily seen
that the rate of convergence of $S_n^{\lambda}(f)$ is slower than
that of $\mathcal{B}_n(f)$ by a factor of $n^{\lambda}$ for
$\lambda>0$ and is the same as that of $S_n^{\lambda}(f)$ for $-1/2<
\lambda<0$, which fully explains the convergence behavior of
$S_n^{\lambda}(f)$ illustrated in Figures \ref{fig:ExamI} and
\ref{fig:ExamII}.
\end{remark}

\begin{remark}
Polynomial interpolation in the zeros of Gegenbauer polynomials is
also a powerful approach for approximating analytic functions. When
the interpolation nodes are the zeros of $C_{n+1}^{\lambda}(x)$, it
has been shown in (Xie \emph{et al.}, 2013, Theorem~4.1) that the
rate of convergence of Gegenbauer interpolation in the maximum norm
is $O(n^{\lambda} \rho^{-n})$ for $\lambda>0$ and is $O(\rho^{-n})$
if $-1/2<\lambda<0$. Comparing this with Theorem \ref{thm:RateAnal},
we see that Gegenbauer interpolation and projection of the same
degree possess the same convergence rate.
\end{remark}

\section{Optimal rates of convergence of Gegenbauer projections for
piecewise analytic functions}\label{sec:Piecewise} In this section
we study optimal rates of convergence of Gegenbauer projections for
piecewise analytic functions of class $C^{m-1}(\Omega)$ with
$m\in\mathbb{N}$. Throughout this paper, we denote by $K$ a generic
positive constant independent of $n$ which may take different values
at different places.

We first introduce the definitions of piecewise analytic functions
and the space of piecewise analytic functions of class
$C^{m-1}(\Omega)$.
\begin{definition}\label{def:PiecewiseAnal}
Let $m$ be a positive integer.
\begin{itemize}
\item[(i)] A function $f$ is said to be piecewise analytic on $\Omega$ if there
exists a set of distinct points $\{\xi_1,\ldots,\xi_{\ell}\}$ with
each $\xi_k\in(-1,1)$ and $\xi_k<\xi_{k+1}$ for $k=1,\ldots,\ell-1$
and $\ell\in\mathbb{N}$, such that the restriction of $f$ to each of
the intervals $[-1,\xi_1]$,$[\xi_1,\xi_2]$,$\ldots$,$[\xi_{\ell},1]$
has an analytic continuation to a neighborhood of this closed
interval, but $f$ itself is not analytic at each point of
$\{\xi_1,\ldots,\xi_{\ell}\}$. Moreover, we call these points
$\{\xi_1,\ldots,\xi_{\ell}\}$ the singularities of $f$.

\item[(ii)] The space of piecewise analytic functions of
class $C^{m-1}(\Omega)$ is defined to be the set of piecewise
analytic functions on $\Omega$ satisfying $f\in C^{m-1}(\Omega)$.
\end{itemize}
\end{definition}

With the above definitions, it is easily verified that these test
functions in \eqref{eq:TestFun2} are piecewise analytic functions of
class $C^{m-1}(\Omega)$ with $m=4,5,3$, respectively. We now
consider optimal convergence rates of Gegenbauer projections for
piecewise analytic functions of class $C^{m-1}(\Omega)$. First of
all, using the integral expression of Gegenbauer coefficients, we
can rewrite the Gegenbauer projection as
\begin{align}
S_n^{\lambda}(f) & = \int_{\Omega} \omega_{\lambda}(t) f(t)
D_n^{\lambda}(x,t) \mathrm{d}t,
\end{align}
where $D_n^{\lambda}(\cdot,\cdot)$ is the reproducing kernel of
Gegenbauer projection defined by
\begin{align}\label{eq:Kn}
D_n^{\lambda}(x,t) &= \sum_{k=0}^{n} \frac{C_k^{\lambda}(x)
C_k^{\lambda}(t)}{h_k^{\lambda}} \nonumber \\
&= \frac{\Gamma(\lambda)^2}{2^{2-2\lambda} \pi}
\frac{\Gamma(n+2)}{\Gamma(n+2\lambda)}
\frac{C_{n+1}^{\lambda}(x)C_{n}^{\lambda}(t) -
C_{n+1}^{\lambda}(t)C_{n}^{\lambda}(x) }{x-t},
\end{align}
and the last equation follows from the Christoffel-Darboux formula
of Gegenbauer polynomials.

The following refined estimates for the reproducing kernel will be
useful.
\begin{lemma}\label{lem:Kn}
Let $|x|\leq1$. Then, for $\lambda\neq0$ and large $n$,
\begin{enumerate}
\item[\rm (i)] If $|t|\leq1$, it holds that $|D_n^{\lambda}(x,t)| \leq K
n^{2 \max\{\lambda,0\} +1}$.

\item[\rm (ii)] If $|t|\leq 1-\varepsilon$ with $\varepsilon\in(0,1)$,
it holds that $|D_n^{\lambda}(x,t)| \leq K
n^{\max\left\{\lambda,1\right\}}$.
\end{enumerate}
\end{lemma}
\begin{proof}
We first consider part (i). From Lemma \ref{lem:Bound} we see that
\begin{align}\label{eq:AsyMaxGen}
\max_{|x|\leq1}|C_n^{\lambda}(x)| = \left\{
            \begin{array}{ll}
{\displaystyle  O(n^{2\lambda-1}) }, & \hbox{$\lambda>0$,}   \\[8pt]
{\displaystyle  O(n^{\lambda-1}) }, & \hbox{$-1/2<\lambda<0$.}
            \end{array}
            \right.
\end{align}
Moreover, using \eqref{eq:AsyGAMMA} we have
$h_n^{\lambda}=O(n^{2\lambda-2})$. Combining these estimates we find
that
\begin{align}
|D_n^{\lambda}(x,t)| \leq \sum_{k=0}^{n} \frac{|C_k^{\lambda}(x)
C_k^{\lambda}(t)|}{h_k^{\lambda}} = \sum_{k=0}^{n} O(k^{2
\max\{\lambda,0\}}) = O(n^{2 \max\{\lambda,0\}+1}). \nonumber
\end{align}
This proves part (i). To prove part (ii), we distinguish two cases:
$|x-t|<\varepsilon/2$ and $|x-t|\geq\varepsilon/2$. For the case
$|x-t|<\varepsilon/2$, it is easily verified that $|x|\leq
1-\varepsilon/2$. Recall from Szeg\H{o} (1939) that
$|C_n^{\lambda}(x)| = O(n^{\lambda-1})$ for $x\in(-1,1)$, we obtain
\begin{align}
\max_{\substack{|x|\leq 1-\varepsilon/2\\ |t|\leq 1 - \varepsilon}}
\frac{|C_k^{\lambda}(x) C_k^{\lambda}(t)|}{h_k^{\lambda}} = O(1),
\nonumber
\end{align}
and thus
\begin{align}
|D_n^{\lambda}(x,t)| \leq \sum_{k=0}^{n} \frac{|C_k^{\lambda}(x)
C_k^{\lambda}(t)|}{h_k^{\lambda}} = \sum_{k=0}^{n} O(1) = O(n).
\nonumber
\end{align}
Next, we consider the case $|x-t|\geq \varepsilon/2$. Combining the
estimate $\max_{|t|\leq1-\varepsilon} |C_n^{\lambda}(t)| =
O(n^{\lambda-1})$ with \eqref{eq:AsyMaxGen}, and the last equality
in \eqref{eq:Kn}, we immediately infer that
\begin{align}
|D_n^{\lambda}(x,t)| = O(n^{\max\{\lambda,0\}}). \nonumber
\end{align}
A combination of the above two estimates gives part (ii). This
completes the proof.
\end{proof}

Now, we prove the main result of this section.
\begin{theorem}\label{thm:PieceRate}
If $f$ belongs to the space of piecewise analytic functions of class
$C^{m-1}(\Omega)$ for some $m\in\mathbb{N}$. Then, for $\lambda<m+1$
and $n\gg1$,
\begin{align}\label{eq:DiffRate}
\|f - S_n^{\lambda}(f)\|_{\infty} \leq K \left\{
            \begin{array}{ll}
{\displaystyle n^{-m} },   & \hbox{$\lambda\leq1$,}   \\[8pt]
{\displaystyle n^{-m-1+\lambda} }, & \hbox{$\lambda>1$.}
            \end{array}
            \right.
\end{align}
Moreover, the convergence rates on the right-hand side of
\eqref{eq:DiffRate} are optimal in the sense that they can not be
improved further.
\end{theorem}
\begin{proof}
Assume that $\{\xi_1,\ldots,\xi_{\ell}\}$, with $\ell\in\mathbb{N}$,
are the singularities of $f$. For every $\eta\in(0,1)$, we know from
Saff \& Totik (1989) that there exists a polynomial $\psi_n$ of
degree $n$ such that
\begin{align}\label{eq:SaffBound}
|f(x) - \psi_n(x) | \leq \frac{C}{n^{m}} \exp\left(-\kappa(n
\mathrm{d}(x))^{\eta} \right), \quad \forall x\in\Omega,
\end{align}
where $\mathrm{d}(x)=\min_{1\leq k\leq\ell}|x-\xi_k|$ and $C,\kappa$
are some positive constants. Recall that $S_n^{\lambda}(f) \equiv f$
whenever $f\in\mathcal{P}_n$, we immediately obtain
\begin{align}\label{eq:ModFunS1}
|f(x) - S_n^{\lambda}(f,x)| &= |f(x) - \psi_n(x) -
S_n^{\lambda}(f-\psi_n,x)|
\nonumber \\
&\leq |f(x) - \psi_n(x)| + |S_n^{\lambda}(f-\psi_n,x)|
\nonumber \\
&\leq \frac{C}{n^{m}} \exp\left(-\kappa (n \mathrm{d}(x))^{\eta}
\right) \nonumber \\
&~~~~~~ + \frac{C}{n^{m}} \int_{\Omega} \exp\left(-\kappa (n
\mathrm{d}(t))^{\eta} \right) \omega_{\lambda}(t) |
D_n^{\lambda}(x,t) | \mathrm{d}t.
\end{align}
We now consider the estimate of the last integral in
\eqref{eq:ModFunS1}. For simplicity of notation we denote it by $I$.
Moreover, let $\Omega_k=[\xi_k-\gamma,\xi_k+\gamma]$, where
$k=1,\ldots,\ell$ and $\gamma>0$ is chosen such that these
subintervals $\Omega_1,\ldots,\Omega_{\ell}\subset(-1,1)$ are
pairwise disjoint 
and thus,
\begin{align}\label{eq:ModFunS2}
I &= \sum_{k=1}^{\ell} \int_{\Omega_k} \exp\left(-\kappa
(n\mathrm{d}(t))^{\eta}\right) \omega_{\lambda}(t)
|D_n^{\lambda}(x,t)| \mathrm{d}t
\nonumber \\
&~~~~~~~~ + \int_{\Omega\backslash\bigcup_{k=1}^{\ell} \Omega_k }
\exp\left(-\kappa (n\mathrm{d}(t))^{\eta}\right) \omega_{\lambda}(t)
|D_n^{\lambda}(x,t)| \mathrm{d}t.
\end{align}
Let $I_1$ and $I_2$ denote the first and second terms on the
right-hand side of \eqref{eq:ModFunS2}, respectively. For $I_1$,
notice that $\mathrm{d}(t)=|t-\xi_k|$ whenever $t\in\Omega_{k}$, and
thus from Lemma \ref{lem:Kn} we have
\begin{align}\label{eq:ModFunS3}
I_1 &\leq \max_{t\in\bigcup_{k=1}^{\ell}\Omega_{k}}
|\omega_{\lambda}(t)| K n^{\max\{\lambda,1\}} \sum_{k=1}^{\ell}
\int_{\Omega_k} \exp\left(-\kappa(n\mathrm{d}(t))^{\eta}\right)
\mathrm{d}t \nonumber \\
&= \max_{t\in\bigcup_{k=1}^{\ell}\Omega_{k}} |\omega_{\lambda}(t)|
(2\ell K) n^{\max\{\lambda,1\}-1} \int_{0}^{n\gamma}
\exp\left(-\kappa \nu^{\eta}\right)
\mathrm{d}\nu \nonumber \\
&= O(n^{\max\{\lambda,1\}-1}),
\end{align}
where we have applied the change of variable $t=\xi_k+\nu/n$ in the
second step. For $I_2$, notice that $\mathrm{d}(t)\geq\gamma$
whenever $t\in\Omega\backslash\bigcup_{k=1}^{\ell}\Omega_k$, and
using Lemma \ref{lem:Kn} again, we obtain
\begin{align}\label{eq:ModFunS4}
I_2 &\leq \exp\left(-\kappa (n\gamma)^{\eta}\right) \int_{\Omega}
\omega_{\lambda}(t) | D_n^{\lambda}(x,t) | \mathrm{d}t \nonumber \\
&\leq \exp\left(-\kappa(n\gamma)^{\eta}\right) K
n^{2\max\{\lambda,0\}+1}
\int_{\Omega} \omega_{\lambda}(t) \mathrm{d}t \nonumber \\
&= O(\exp\left(-\kappa(n\gamma)^{\eta}\right)
n^{2\max\{\lambda,0\}+1}).
\end{align}
Combining \eqref{eq:ModFunS1}, \eqref{eq:ModFunS2},
\eqref{eq:ModFunS3} and \eqref{eq:ModFunS4}
gives \eqref{eq:DiffRate}. 

We now turn to prove the optimality of the convergence rates on the
right-hand side of \eqref{eq:DiffRate}. Recall that $\|f -
\mathcal{B}_n(f)\|_{\infty}=O(n^{-m})$ (see, e.g., Timan, 1963,
Chapter 7). In the case $\lambda\leq1$, the rate of convergence of
$S_n^{\lambda}(f)$ is obviously optimal since it is the same as that
of $\mathcal{B}_n(f)$. In the case $\lambda>1$, the predicted
convergence rate is $\|f -
S_n^{\lambda}(f)\|_{\infty}=O(n^{-m-1+\lambda})$. To show the
optimality of this rate, we consider a specific example
$f(x)=(x)_{+}^5$, which corresponds to $m=5$. In view of (Olver
\emph{et al.}, 2010, Equation~(18.17.37)), the Gegenbauer
coefficients of this function are given by
\begin{align}
a_k^{\lambda} &= \frac{15}{8} \frac{\Gamma(\lambda) (k +
\lambda)}{\Gamma(\lambda + \frac{k+7}{2}) \Gamma(\frac{7-k}{2})},
\quad k=0,1,\ldots,
\end{align}
from which we can see that $a_{2k+1}^{\lambda}=0$ for $k\geq3$. For
$k\geq6$ is even, we have, using \eqref{eq:reflection} and
\eqref{eq:AsyGAMMA},
\begin{align}\label{eq:diffS1}
a_{k}^{\lambda} &= (-1)^{\frac{k}{2}+1}
\frac{15\Gamma(\lambda)}{8\pi} \frac{(k+\lambda)
\Gamma(\frac{k-5}{2})}{\Gamma(\frac{k+7}{2}+\lambda)} =
(-1)^{\frac{k}{2}+1} \frac{15\Gamma(\lambda)}{4\pi}
\left(\frac{k}{2} \right)^{-\lambda-5} + O(k^{-\lambda-6}).
\end{align}
Now we consider the error estimate of $S_n^{\lambda}(f)$ at $x=1$.
Assume that $n\geq6$ is a large even integer, using
\eqref{eq:diffS1} and the asymptotic estimate $C_{k}^{\lambda}(1) =
k^{2\lambda-1}/\Gamma(2\lambda)+ O(k^{2\lambda-2})$, we obtain that
\begin{align}
f(1) - S_n^{\lambda}(f,1) &= \sum_{k=1}^{\infty}
a_{n+2k}^{\lambda} C_{n+2k}^{\lambda}(1) \nonumber \\
&\sim (-1)^{\frac{n}{2}+1} \frac{15\Gamma(\lambda)
2^{\lambda+3}}{\pi \Gamma(2\lambda)} n^{\lambda-6}
\sum_{k=1}^{\infty} (-1)^k \left(1 +
\frac{2k}{n} \right)^{\lambda-6} \nonumber \\
&= O(n^{\lambda-6}),  \qquad  n\gg1, \nonumber
\end{align}
where in the last step we have used the fact that the alternating
series is always bounded for $\lambda<6$. Similarly, it is not
difficult to show that $f(1)-S_n^{\lambda}(f,1)=O(n^{\lambda-6})$ if
$n\geq6$ is a large odd integer. Since $\|f -
S_n^{\lambda}(f)\|_{\infty} \geq |f(1) - S_n^{\lambda}(f,1)|$, we
can conclude that the predicted rate $\|f -
S_n^{\lambda}(f)\|_{\infty} = O(n^{\lambda-6})$ is optimal. This
completes the proof.
\end{proof}

In order to verify the convergence rates predicted by Theorem
\ref{thm:PieceRate}, we consider the test functions in
\eqref{eq:TestFun2}, which correspond to $m=4,5,3$, respectively.
From Theorem \ref{thm:PieceRate} we know that the predicted rate of
$S_n^{\lambda}(f_4)$ is $O(n^{-4})$ if $\lambda\leq1$ and is
$O(n^{\lambda-5})$ if $\lambda>1$, and the predicted rate of
$S_n^{\lambda}(f_5)$ is $O(n^{-5})$ if $\lambda\leq1$ and is
$O(n^{\lambda-6})$ if $\lambda>1$, and the predicted rate of
$S_n^{\lambda}(f_6)$ is $O(n^{-3})$ if $\lambda\leq1$ and is
$O(n^{\lambda-4})$ if $\lambda>1$. For each $f_j$, we test the
convergence rates of $S_n^{\lambda}(f_j)$ with four values of
$\lambda$ and they are displayed in Figure \ref{fig:ExamV}. Clearly,
for each $\lambda$, we see that the actual convergence rate of
$S_n^{\lambda}(f)$ coincides quite well with the predicted rate.
Moreover, these results also explain the observations in Figures
\ref{fig:ExamIII} and \ref{fig:ExamIV} since the convergence rates
of $\mathcal{B}_n(f)$ for $f_4,f_5$ and $f_6$ are
$O(n^{-4}),O(n^{-5})$ and $O(n^{-3})$, respectively.

\begin{figure}[ht]
\centering
\includegraphics[width=.325\textwidth,height=5cm]{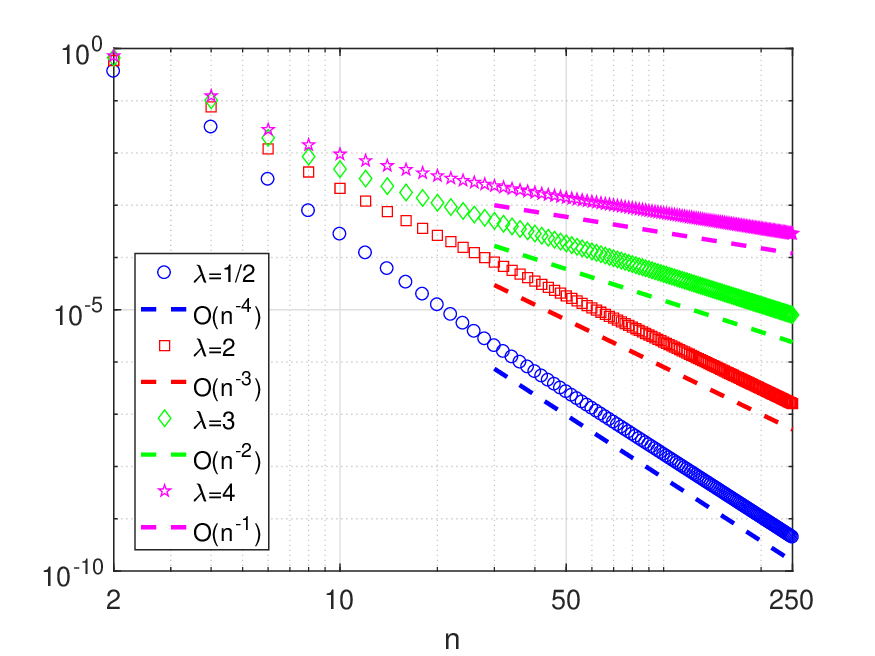}~
\includegraphics[width=.325\textwidth,height=5cm]{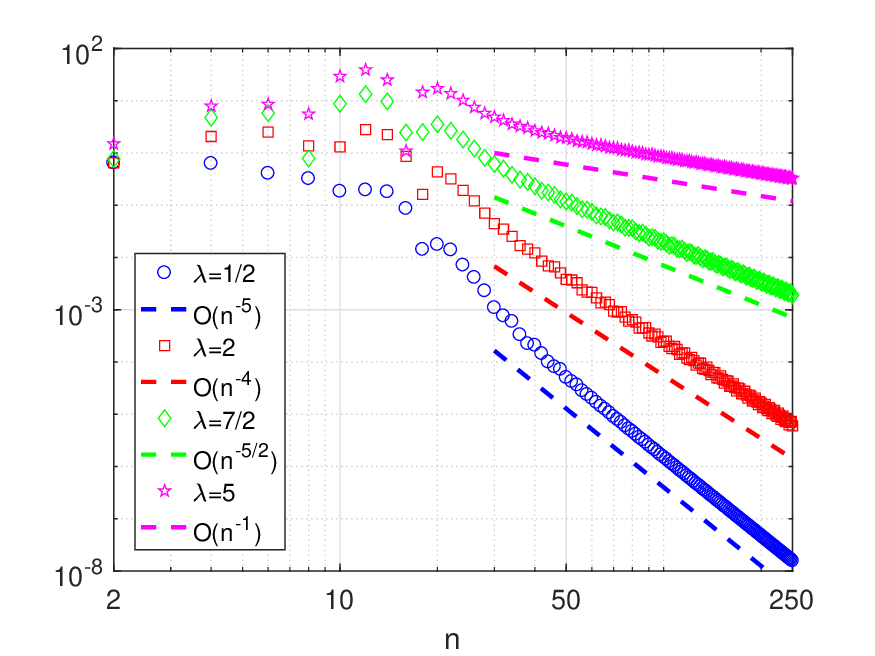}
\includegraphics[width=.325\textwidth,height=5cm]{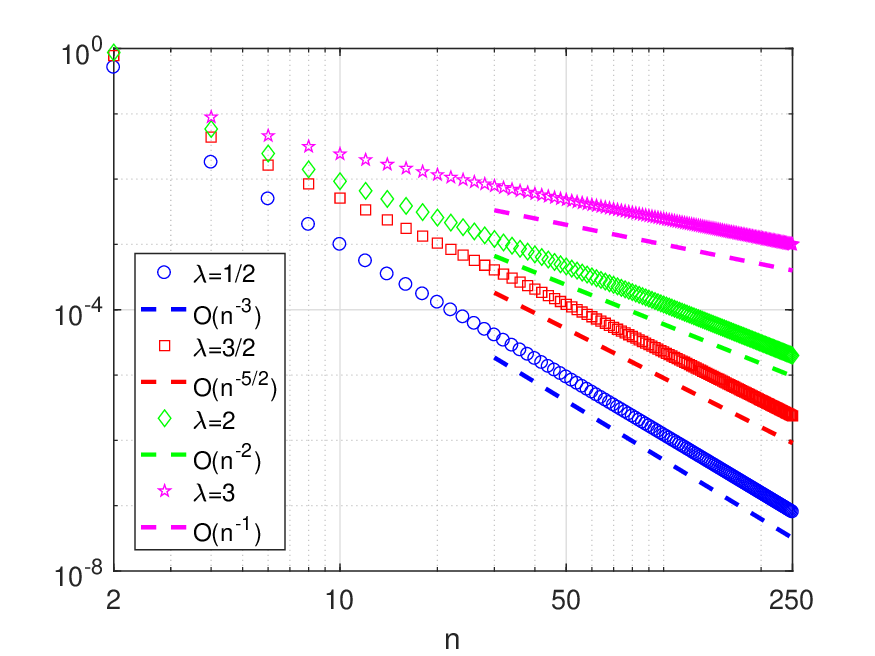}
\caption{The maximum errors of $S_n^{\lambda}(f_4)$ (left),
$S_n^{\lambda}(f_5)$ (middle) and $S_n^{\lambda}(f_6)$ (right).
Dashed lines indicate the convergence rates predicted by Theorem
\ref{thm:PieceRate}.} 
\label{fig:ExamV}
\end{figure}

\section{Optimal rates of convergence of Gegenbauer projections for
functions with algebraic singularities}\label{sec:fractional} In
this section we consider optimal rates of convergence of Gegenbauer
projections for functions with algebraic singularities.
Specifically, we divide our discussion into two cases: (i) functions
with an interior singularity; (ii) functions with an endpoint
singularity. For ease of clarity and conciseness, we restrict
ourselves to the following model function
\begin{align}\label{def:Model}
f(x)=|x-\theta|^{\alpha},
\end{align}
where $\theta\in\Omega$ and $\alpha>0$ is not an even integer
whenever $\theta\in(-1,1)$ and is not an integer whenever
$\theta=\pm1$. The convergence rate results will shed light on the
study of more complicated functions with algebraic singularities.

\begin{remark}
Although we restrict ourselves to the model function
\eqref{def:Model}, the extension to more general functions involving
one or more singularities of $|x-\theta|^{\alpha}$-type, such as $
f(x) = \sum_{k=1}^{m} |x-\theta_k|^{\alpha_k} g_k(x)$, where
$-1\leq\theta_1<\cdots<\theta_m\leq1$ and $\alpha_k>0$ are not
integers and $g_k(x)$ are sufficiently smooth, is straightforward.
Moreover, for functions of the form $f(x) = g(x)\prod_{k=1}^{m}
|x-\theta_k|^{\alpha_k}$, where $g(x)$ is sufficiently smooth, by
noticing that they can also be decomposed into a sum of $m$
functions and each function contains exactly one singularity of
$|x-\theta|^{\alpha}$-type (Tuan \& Elliott, 1972), our analysis can
also be applied to handle such functions.
\end{remark}

\subsection{The case $\theta\in(-1,1)$}
In the case where $\alpha$ is an odd integer, note that $f$ belongs
to the space of piecewise analytic functions of class
$C^{\alpha-1}(\Omega)$, and thus the optimal rate of convergence of
$S_n^{\lambda}(f)$ follows immediately from Theorem
\ref{thm:PieceRate}. In the case where $\alpha$ is not an integer,
however, Theorem \ref{thm:PieceRate} can not be used and a new
approach for error estimates of $S_n^{\lambda}(f)$ should be
developed.

Before we proceed, let us consider the location of the maximum error
of $S_n^{\lambda}(f)$. In the particular case $\lambda=1/2$, which
corresponds to Legendre projections, it has been observed in Wang
(2021) that the maximum error is attained at $x=\theta$. For the
Gegenbauer case, however, the situation may be complicated and it is
highly interesting to clarify the dependence of the location of the
maximum error on the parameter $\lambda$. To gain some insight, we
plot in Figure \ref{fig:ExamVI} the pointwise error of
$S_n^{\lambda}(f)$ with three values of $\lambda$. Clearly, we
observe that, for $\lambda$ greater than a critical value, the
location of the maximum error of $S_n^{\lambda}(f)$ will jump from
$x=\theta$ to one of the endpoints $x=1$ or $x=-1$. Motivated by
this observation, we shall consider the pointwise error of
$S_n^{\lambda}(f)$ and then clarify the maximum error of
$S_n^{\lambda}(f)$.

\begin{figure}[ht]
\centering
\includegraphics[width=4.9cm,height=4.3cm]{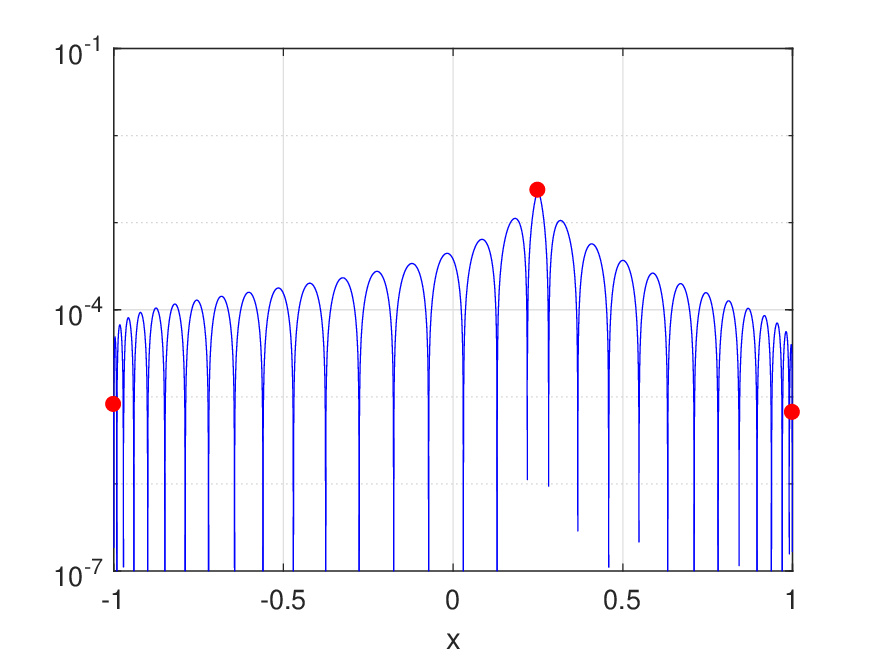}~
\includegraphics[width=4.9cm,height=4.3cm]{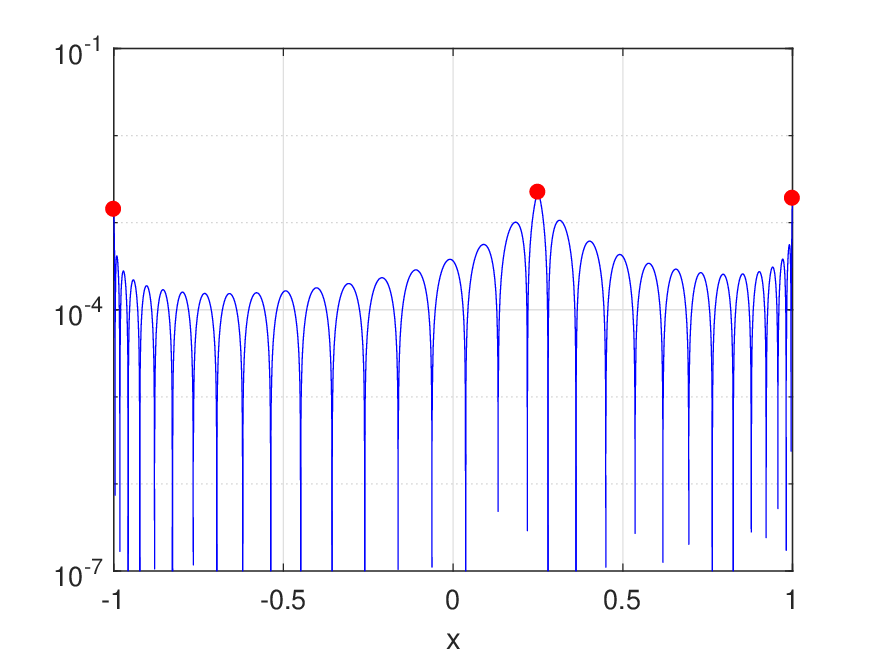}~
\includegraphics[width=4.9cm,height=4.3cm]{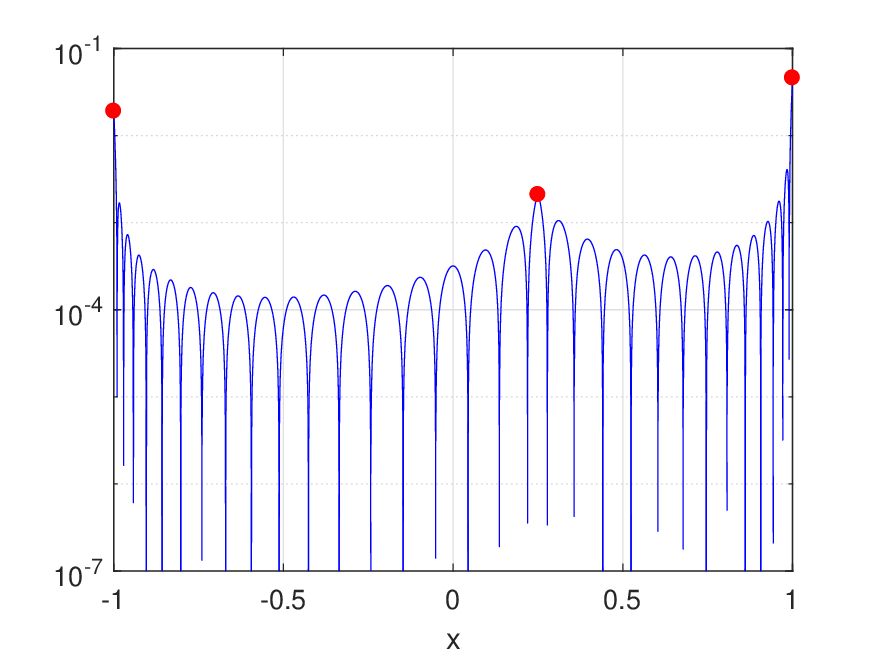}
\caption{Pointwise error of $S_n^{\lambda}(f)$ for $\lambda=-2/5$
(left), $\lambda=3/4$ (middle) and $\lambda=2$ (right). Here
$f(x)=|x-1/4|^{3/2}$ and $n=30$. These red points are the errors of
$S_n^{\lambda}(f)$ at the critical points, i.e., $x=\theta,\pm1$.}
\label{fig:ExamVI}
\end{figure}

We start with the following result.
\begin{lemma}\label{lem:IntSing}
Let $f$ be defined by \eqref{def:Model} with $\theta\in(-1,1)$ and
let $\alpha>0$ be not an even integer.
\begin{itemize}
\item[\rm (i)] For each $k\geq\alpha+1$\footnote{This condition is imposed
here due to the definition of generalized Gegenbauer functions
proposed in (Liu \emph{et al.}, 2019, Definition~2.1). However,
numerical tests show that the formula \eqref{eq:ExpFormula} is valid
for all $k\geq0$. To keep the proof concise, we will not pursue this
here.},
\begin{align}\label{eq:ExpFormula}
a_k^{\lambda} &= \omega_{\lambda+\alpha+1}(\theta)
\frac{\Gamma(\lambda) \Gamma(\alpha+1) (k+\lambda)}{2^{1+\alpha}
\Gamma(\lambda+\alpha+\frac{3}{2}) \sqrt{\pi}} \left( {}_2\mathrm{
F}_1\left[\begin{matrix} \alpha+1-k, & k+2\lambda+\alpha+1
\\   \alpha + \lambda + \frac{3}{2} \hspace{-1cm} &\end{matrix} ;~ \frac{1-\theta}{2} \right]
\right. \nonumber \\[8pt]
&~~~~~~~~~~ \left.  + ~(-1)^k {}_2\mathrm{ F}_1\left[\begin{matrix}
\alpha+1-k, & k+2\lambda+\alpha+1
\\   \alpha + \lambda + \frac{3}{2} \hspace{-1cm} &\end{matrix} ; ~ \frac{1+\theta}{2} \right]
\right).
\end{align}

\item[\rm (ii)] As $k\rightarrow\infty$,
\begin{align}\label{eq:AsymInt}
a_k^{\lambda} &=
-\omega_{\frac{\lambda+\alpha+1}{2}}(\theta)\sin\left(\frac{\alpha\pi}{2}\right)
\frac{2^{1+\lambda}\Gamma(\lambda)\Gamma(\alpha+1)}{\pi
k^{\alpha+\lambda}} \cos\left(2(k+\lambda)\phi(\theta) -
\frac{\lambda\pi}{2} \right) \nonumber \\
&~~~~~ + O(k^{-\alpha-\lambda-1}),
\end{align}
where $\phi(\theta)=\arccos(\sqrt{(1+\theta)/2})$.
\end{itemize}
\end{lemma}
The proof of Lemma \ref{lem:IntSing} is postponed to Appendix A.

\begin{remark}
An immediate corollary of Lemma \ref{lem:IntSing} is the comparison
of decay rates of Chebyshev and Legendre coefficients, which was
studied in Boyd \& Petschek (2014) and Wang (2016). More
specifically, let $k\geq1$ and let $a_k^L$ and $a_k^C$,
respectively, denote the $k$th Legendre and Chebyshev coefficients
of $f$ defined by \eqref{def:Model}, i.e.,
\begin{align}\label{eq:ChebLeg}
a_k^L = \frac{2k+1}{2} \int_{\Omega} f(x) P_k(x) \mathrm{d}x, \quad
a_k^C = \frac{2}{\pi} \int_{\Omega} \frac{f(x) T_k(x)}{\sqrt{1-x^2}}
\mathrm{d}x.
\end{align}
It has been observed in the right panel of Figure 7 in Wang (2016)
that $a_k^C$ decays faster than $a_k^L$ by a factor of $O(k^{1/2})$
and the sequence $\{a_k^L/a_k^C k^{-1/2}\}$ oscillates around a
finite value as $k\rightarrow\infty$. However, a theoretical
explanation for this observation is still lacking. To clarify this
issue, using \eqref{eq:AsymInt} and \eqref{eq:GegChebT}, after some
simplifications, we obtain that
\begin{align}
\frac{a_k^{L}}{a_k^{C}} &= \omega_{\frac{3}{4}}(\theta)
\frac{\cos\left((2k+1)\phi(\theta) -
\frac{\pi}{4}\right)}{\cos\left(2k\phi(\theta) \right)}
\left(\frac{\pi k}{2}\right)^{1/2} + O(k^{-1/2}).
\end{align}
Consequently, we can see that the sequence $\{a_k^L/a_k^C
k^{-1/2}\}$ oscillates around a finite value as $k\rightarrow\infty$
whenever $\theta\neq0$ and tends to the constant $(\pi/2)^{1/2}$
whenever $\theta=0$.
\end{remark}

The following lemma will also be useful.
\begin{lemma}\label{lem:OscSeries}
Let $\nu\in\mathbb{R}$ 
and $\nu(\mathrm{mod}~2\pi)\neq0$. Then, for $\mu<0$, it holds that
\begin{align}\label{eq:OscSeries}
\chi_{\mu,\nu}(n) := \sum_{k=n}^{\infty} e^{ik\nu}k^{\mu} =O(n^\mu),
\quad n\rightarrow\infty.
\end{align}
\end{lemma}
\begin{proof}
For $\mu<-1$, the desired estimate follows immediately from (Olver,
1974, Equation~(5.10)). For $-1\leq\mu<0$, using the identity
(Olver, 1974, Equation~(5.09)), we have that
\begin{align}
\chi_{\mu,\nu}(n) = \frac{(-\mu) e^{i\nu}}{e^{i\nu}-1}
\chi_{\mu-1,\nu}(n) - \frac{e^{i n\nu}}{e^{i\nu}-1} n^{\mu} +
O(n^{\mu-1}). \nonumber
\end{align}
Since $\chi_{\mu-1,\nu}(n)=O(n^{\mu-1})$ in this case, the desired
estimate follows immediately.
\end{proof}

The main theorem in this subsection is now given as follows.
\begin{theorem}\label{thm:InterRate}
Let $f$ be defined by \eqref{def:Model} with $\theta\in(-1,1)$ and
let $\alpha>0$ be not an even integer. Then, for $\lambda<\alpha+1$
and $n\gg1$, it holds that
\begin{itemize}
\item[\rm (i)] The maximum error of $S_n^{\lambda}(f)$ satisfies
\begin{align}\label{eq:MaxErrInt}
\|f - S_n^{\lambda}(f) \|_{\infty} = \left\{
            \begin{array}{ll}
{\displaystyle O(n^{-\alpha}) },           & \hbox{$\lambda\leq1$,}   \\[8pt]
{\displaystyle O(n^{-\alpha-1+\lambda})},  & \hbox{$\lambda>1$.}
            \end{array}
            \right.
\end{align}

\item[\rm (ii)] For $x\in\Omega$, the pointwise error estimate of
$S_n^{\lambda}(f)$ is
\begin{align}\label{eq:PointwiseInt}
|f(x) - S_n^{\lambda}(f,x)| = \left\{
\begin{array}{ll}
{\displaystyle O(n^{-\alpha-1+\lambda})}, & \hbox{$x=\pm1$,}   \\[8pt]
{\displaystyle O(n^{-\alpha}) }, & \hbox{$x=\theta$,}   \\[8pt]
{\displaystyle O(n^{-\alpha-1}) }, &
\hbox{$x\in(-1,\theta)\cup(\theta,1)$.}
            \end{array}
            \right.
\end{align}

\end{itemize}
\end{theorem}
\begin{proof}
We only consider the proof of part (ii) since part (i) is a direct
consequence of part (ii). We start with the error estimate of
$S_n^{\lambda}(f)$ at $x=1$. From Lemma \ref{lem:IntSing} and the
fact that
$C_k^{\lambda}(1)=k^{2\lambda-1}/\Gamma(2\lambda)+O(k^{2\lambda-2})$,
we have
\begin{align}
f(1) - S_n^{\lambda}(f,1) &=
-\omega_{\frac{\lambda+\alpha+1}{2}}(\theta)
\sin\left(\frac{\alpha\pi}{2}\right)
\frac{2^{1+\lambda}\Gamma(\lambda)\Gamma(\alpha+1)}{\pi\Gamma(2\lambda)}
\nonumber \\
&~~~~~ \times \sum_{k=n+1}^{\infty} \left(
\frac{\cos\left(2(k+\lambda)\phi(\theta) - \frac{\lambda\pi}{2}
\right)}{k^{\alpha+1-\lambda}} + O(k^{-\alpha+\lambda-2}) \right).
\nonumber
\end{align}
Furthermore, we note that $\phi(\theta)\in(0,\pi/2)$ and
\begin{align}
\sum_{k=n+1}^{\infty} \frac{\cos\left(2(k+\lambda)\phi(\theta) -
\frac{\lambda\pi}{2} \right)}{k^{\alpha+1-\lambda}} &=
\cos\left(2\lambda\phi(\theta)-\frac{\lambda\pi}{2}\right)
\sum_{k=n+1}^{\infty} \frac{\cos\left(2k\phi(\theta)
\right)}{k^{\alpha+1-\lambda}} \nonumber \\
&~ - \sin\left(2\lambda\phi(\theta)-\frac{\lambda\pi}{2}\right)
\sum_{k=n+1}^{\infty} \frac{\sin\left(2k\phi(\theta)
\right)}{k^{\alpha+1-\lambda}}, \nonumber
\end{align}
and therefore, by Lemma \ref{lem:OscSeries}, these two sums on the
right-hand side behave like $O(n^{-\alpha-1+\lambda})$. This proves
the error estimate of $S_n^{\lambda}(f)$ at $x=1$. The error
estimate of $S_n^{\lambda}(f)$ at $x=-1$ can be proved in a similar
way and we omit the details.

Next, we consider the error estimate of $S_n^{\lambda}(f)$ at
$x\in(-1,1)$. For notational simplicity, we set $x=\cos\zeta$, where
$\zeta\in(0,\pi)$. According to Theorem~8.21.8 in Szeg\H{o} (1939),
\begin{align}\label{eq:GegenbauerAsym}
C_k^{\lambda}(x) = \frac{(1-x^2)^{-\lambda/2}}{\Gamma(\lambda)}
\left(\frac{2}{k}\right)^{1-\lambda} \cos\left((k+\lambda)\zeta -
\frac{\lambda\pi}{2} \right) + O(k^{\lambda-2}).
\end{align}
Combining \eqref{eq:GegenbauerAsym} with \eqref{eq:AsymInt} in Lemma
\ref{lem:IntSing}, after some simplification, we arrive at
\begin{align}
f(x) - S_n^{\lambda}(f,x) &= \sum_{k=n+1}^{\infty} a_k^{\lambda}
C_k^{\lambda}(x) \nonumber \\
&= -2\sin\left(\frac{\alpha\pi}{2}\right)
\omega_{\frac{\lambda+\alpha+1}{2}}(\theta) \Gamma(\alpha+1)
\frac{(1-x^2)^{-\lambda/2}}{\pi} \nonumber \\
&~~~~ \times \left[ \sum_{k=n+1}^{\infty}
\frac{\cos\left((k+\lambda)(2\phi(\theta)-\zeta) \right) +
\cos\left((k+\lambda)(2\phi(\theta)+\zeta)-\lambda\pi \right)
}{k^{\alpha+1}} \right]    \nonumber \\
&~~~~ + O(n^{-\alpha-1}).  \nonumber
\end{align}
We denote with $J$ the term inside the bracket on the right-hand
side of the above equation and it is easily seen that the error
estimate of $S_n^{\lambda}(f)$ is completely determined by the
asymptotic behavior of $J$. We now consider the error estimate of
$S_n^{\lambda}(f)$ at the singularity $x=\theta$. In this case, it
is easily checked that $\zeta=\arccos\theta=2\phi(\theta)$, and thus
\begin{align}
J &= \sum_{k=n+1}^{\infty} \frac{1 +
\cos\left((k+\lambda)(2\zeta)-\lambda\pi \right)}{k^{\alpha+1}}
\nonumber \\
&= \sum_{k=n+1}^{\infty} \frac{1}{k^{\alpha+1}} +
\cos((\pi-2\zeta)\lambda) \sum_{k=n+1}^{\infty}
\frac{\cos(2k\zeta)}{k^{\alpha+1}} + \sin((\pi-2\zeta)\lambda)
\sum_{k=n+1}^{\infty} \frac{\sin(2k\zeta)}{k^{\alpha+1}}. \nonumber
\end{align}
Clearly, the first sum behaves like $O(n^{-\alpha})$ and the last
two sums, in view of Lemma \ref{lem:OscSeries}, behave like
$O(n^{-\alpha-1})$. Hence, we conclude that $J=O(n^{-\alpha})$ and
this proves the error estimate of $S_n^{\lambda}(f)$ at $x=\theta$.
Finally, we consider the error estimate of $S_n^{\lambda}(f)$ at
$x\in(-1,1)\backslash\{\theta\}$. In this case, we note that
\begin{align}
J &= \cos(\lambda(2\phi(\theta)-\zeta)) \sum_{k=n+1}^{\infty}
\frac{\cos(k(2\phi(\theta)+\zeta))}{k^{\alpha+1}} \nonumber \\
& - \sin(\lambda(2\phi(\theta)-\zeta)) \sum_{k=n+1}^{\infty}
\frac{\sin(k(2\phi(\theta)+\zeta))}{k^{\alpha+1}} \nonumber \\
& + \cos(\lambda(2\phi(\theta)+\zeta-\pi)) \sum_{k=n+1}^{\infty}
\frac{\cos(k(2\phi(\theta)+\zeta))}{k^{\alpha+1}} \nonumber \\
& - \sin(\lambda(2\phi(\theta)+\zeta-\pi)) \sum_{k=n+1}^{\infty}
\frac{\sin(k(2\phi(\theta)+\zeta))}{k^{\alpha+1}}, \nonumber
\end{align}
and by using Lemma \ref{lem:OscSeries} again and the fact that
$2\phi(\theta)+\zeta\in(0,2\pi)$, these four sums on the right-hand
side all behave like $O(n^{-\alpha-1})$. Therefore, we conclude that
$J=O(n^{-\alpha-1})$ and this proves the error estimate of
$S_n^{\lambda}(f)$ at $x\in(-1,1)\backslash\{\theta\}$. This
completes the proof.
\end{proof}

Several remarks on Theorem \ref{thm:InterRate} are in order.
\begin{remark}
Recall from Timan (1963) that the rate of convergence of
$\mathcal{B}_n(f)$ in the maximum norm is $O(n^{-\alpha})$.
Therefore, the rate of convergence of $S_n^{\lambda}(f)$ is the same
as that of $\mathcal{B}_n(f)$ whenever $-1/2<\lambda\leq1$. For
$\lambda>1$, however, the rate of convergence of $S_n^{\lambda}(f)$
is slower than that of $\mathcal{B}_n(f)$ by a factor of
$n^{\lambda-1}$, which is one power of $n$ better than the result
predicted by \eqref{eq:errorG}.
\end{remark}

\begin{remark}
Pointwise error estimates of Jacobi projections were studied in
Agahanov \& Natanson (1966) in the space
\[
W_{\mu}^{\nu}(\Omega) = \left\{f~\big|~
f,f{'},\cdots,f^{(\nu-1)}\in\mathrm{AC}(\Omega),~
f^{(\nu)}\in\mathrm{H}^{\mu}(\Omega) \right\},
\]
where $\nu\in\mathbb{N}$, $\mu\in[0,1]$ and $\mathrm{AC}(\Omega)$
denotes the space of absolutely continuous functions and
$\mathrm{H}^{\mu}(\Omega)$ denotes the space of H\"{o}lder
continuous function with exponent $\mu$. When restricting their
results to the case of Gegenbauer projections and the model function
\eqref{def:Model}, their results can be written as
\[
|f(x) - S_n^{\lambda}(f,x)| = \left\{
\begin{array}{ll}
{\displaystyle O(n^{\lambda-\alpha})}, & \hbox{$x=\pm1$,}   \\[8pt]
{\displaystyle O(n^{-\alpha}\ln n) },  & \hbox{$x\in(-1,1)$.}
            \end{array}
            \right.
\]
Compared with Theorem \ref{thm:InterRate}, it is clear to see that
our results are sharper.
\end{remark}

\begin{figure}[ht]
\centering
\includegraphics[width=7.5cm,height=6cm]{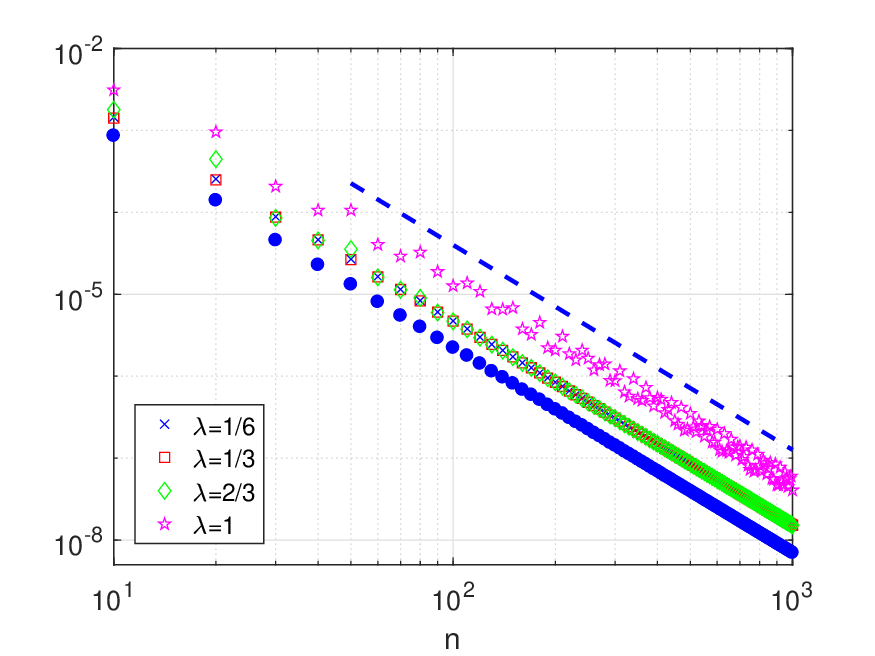}~
\includegraphics[width=7.5cm,height=6cm]{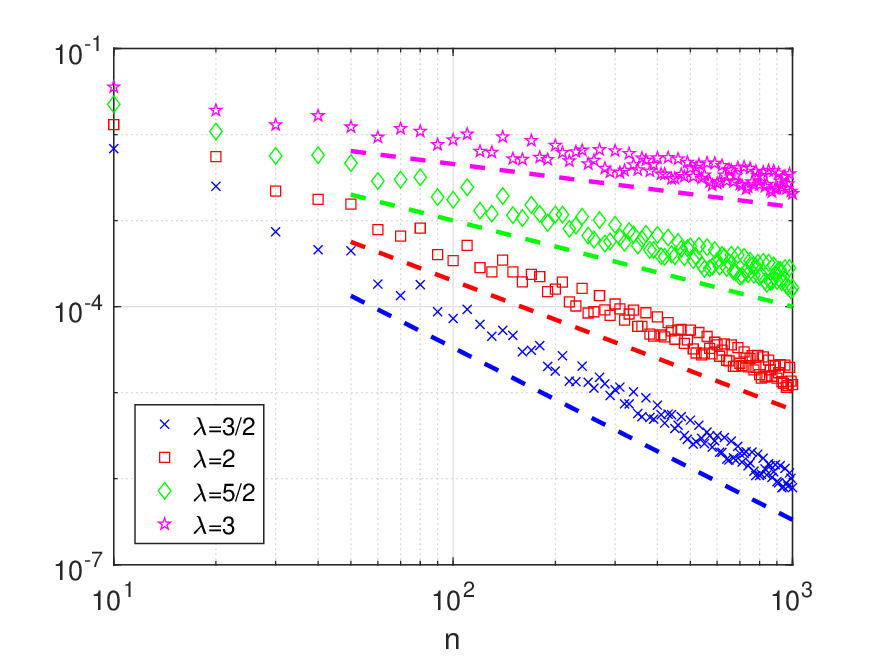}
\caption{The left panel shows the maximum errors of
$\mathcal{B}_n(f)$ (dots) and $S_n^{\lambda}(f)$ for
$\lambda=1/6,1/3,2/3,1$. The right panel shows the maximum errors of
$S_n^{\lambda}(f)$ for $\lambda=3/2,2,5/2,3$. The dashed line in the
left panel is $O(n^{-5/2})$ and these dashed lines in the right
panel indicate the convergence rates predicted by
\eqref{eq:MaxErrInt}. Here $f(x)=|x+0.4|^{5/2}$.}
\label{fig:ExamVIII}
\end{figure}

In Figure \ref{fig:ExamVIII} we illustrate the maximum error of
$S_n^{\lambda}(f)$ for the test function $f(x)=|x+0.4|^{5/2}$. As
expected, the predicted convergence rates by \eqref{eq:MaxErrInt}
agree quite well with the observed convergence rates.

\subsection{The case $\theta=\pm1$}
Error estimates of Gegenbauer projections for functions with
endpoint singularities have been studied in the recent work Xiang \&
Liu (2020) and optimal convergence rates of $S_n^{\lambda}(f)$ in
the maximum norm have been derived based on optimal decay rates of
the Gegenbauer coefficients. Here we revisit
this issue and provide a more thorough insight. 
\begin{theorem}\label{thm:EndRate}
Let $f$ be defined by \eqref{def:Model} with $\theta=\pm1$ and
$\alpha>0$ is not an integer.
\begin{itemize}
\item[\rm (i)] For $\lambda>0$ and $n\geq \lfloor \alpha \rfloor$, the
maximum error of $S_n^{\lambda}(f)$ is attained at $x=\theta$ and
\begin{align}
\|f - S_n^{\lambda}(f) \|_{\infty} = \frac{2^{\alpha}
|\sin(\alpha\pi)| \Gamma(\alpha+\lambda +\frac{1}{2}) \Gamma(\alpha)
}{\pi\Gamma(\lambda+\frac{1}{2}) n^{2\alpha}} + O(n^{-2\alpha-1}).
\end{align}

\item[\rm (ii)] For $\lambda>0$ and large $n$, the pointwise error estimate
is
\begin{align}\label{eq:PointwiseEnd}
|f(x) - S_n^{\lambda}(f,x) | = \left\{
\begin{array}{ll}
{\displaystyle O(n^{-2\alpha}) },           & \hbox{$x=\theta$,}     \\[8pt]
{\displaystyle O(n^{-2\alpha-1}) },         & \hbox{$x=-\theta$,}    \\[8pt]
{\displaystyle O(n^{-2\alpha-\lambda-1}) }, & \hbox{$|x|<|\theta|$.}
            \end{array}
            \right.
\end{align}

\end{itemize}
\end{theorem}
\begin{proof}
We first prove part (i). Using (Gradshteyn \& Ryzhik, 2007,
Equation~(7.311.3)), \eqref{eq:reflection} and
\eqref{eq:duplication}, we can write the Gegenbauer coefficients of
$f$ as
\begin{align}\label{eq:akEnd}
a_k^{\lambda} &= -\theta^{k} \frac{2^{2\lambda+\alpha}
\sin(\alpha\pi) \Gamma(\lambda) \Gamma(\alpha+\lambda +\frac{1}{2})
\Gamma(\alpha+1) (k+\lambda) \Gamma(k-\alpha) }{\pi^{3/2}
\Gamma(k+\alpha+2\lambda+1)}.
\end{align}
An important observation is that, for $k\geq\lfloor \alpha \rfloor +
1$, the sequence $\{a_k^{\lambda}\}$ is a sequence with alternating
sign whenever $\theta=-1$ and is a sequence with constant sign
whenever $\theta=1$. Consequently, for $n\geq \lfloor \alpha
\rfloor$, we can deduce from the symmetry property of
$C_k^{\lambda}(x)$ that
\begin{align}
\|f - S_n^{\lambda}(f) \|_{\infty} &\leq \sum_{k=n+1}^{\infty} |
a_k^{\lambda} | C_k^{\lambda}(|\theta|)  
= |f(\theta) - S_n^{\lambda}(f,\theta)|, \nonumber
\end{align}
which implies that the maximum error of $S_n^{\lambda}(f)$ is
attained at $x=\theta$. Combining this with \eqref{eq:akEnd} and
\eqref{eq:AsyGAMMA} we have
\begin{align}
\|f - S_n^{\lambda}(f)\|_{\infty} &= \frac{2^{\alpha+1}
|\sin(\alpha\pi)| \Gamma(\alpha+\lambda +\frac{1}{2})
\Gamma(\alpha+1)}{\pi\Gamma(\lambda+\frac{1}{2})}
\sum_{k=n+1}^{\infty} \frac{(k+\lambda) \Gamma(k-\alpha)
\Gamma(k+2\lambda) }{\Gamma(k+\alpha+2\lambda+1)
\Gamma(k+1)} \nonumber \\
&= \frac{2^{\alpha+1} |\sin(\alpha\pi)| \Gamma(\alpha+\lambda
+\frac{1}{2}) \Gamma(\alpha+1) }{\pi\Gamma(\lambda+\frac{1}{2})}
\sum_{k=n+1}^{\infty} \left( \frac{1}{k^{2\alpha+1}} + O(k^{-2\alpha-2}) \right) \nonumber \\
&= \frac{2^{\alpha} |\sin(\alpha\pi)| \Gamma(\alpha+\lambda
+\frac{1}{2}) \Gamma(\alpha) }{\pi\Gamma(\lambda+\frac{1}{2})
n^{2\alpha}} + O(n^{-2\alpha-1}).  \nonumber
\end{align}
This proves part (i).

As for part (ii), the pointwise error estimate at $x=\theta$ follows
from part (i) directly and at $x=-\theta$ follows from
\eqref{eq:akEnd} and the symmetry property of Gegenbauer
polynomials. For the case $|x|<|\theta|$, the pointwise error
estimate follows from \eqref{eq:GegenbauerAsym} and
\eqref{eq:akEnd}. This ends the proof.

\end{proof}

Some remarks are in order.
\begin{remark}
From Timan (1963) we know that the rate of convergence of
$\mathcal{B}_n(f)$ is $O(n^{-2\alpha})$. In the case $\lambda<0$,
from \eqref{eq:errorG} and \eqref{eq:Lebesgue} we know that
$S_n^{\lambda}(f)$ converges at the same rate as $\mathcal{B}_n(f)$,
we can thus infer that the rate of convergence of $S_n^{\lambda}(f)$
is $O(n^{-2\alpha})$. In the case $\lambda=0$, from Liu \emph{et
al.} (2019) we know that the rate of convergence of Chebyshev
projection of degree $n$ is also $O(n^{-2\alpha})$. Therefore,
combining these with Theorem \ref{thm:EndRate} we conclude that
$S_n^{\lambda}(f)$ and $\mathcal{B}_n(f)$ converge at the same rate
for all $\lambda>-1/2$.
\end{remark}

\begin{remark}
Observe that the constant in the leading term of $\|f -
S_n^{\lambda}(f) \|_{\infty}$ behaves like $O(\lambda^{\alpha})$ as
$\lambda\rightarrow\infty$, we can deduce that the maximum error of
$S_n^{\lambda}(f)$ will deteriorate as $\lambda$ increases.
\end{remark}

\begin{figure}[ht]
\centering
\includegraphics[width=7.5cm,height=6cm]{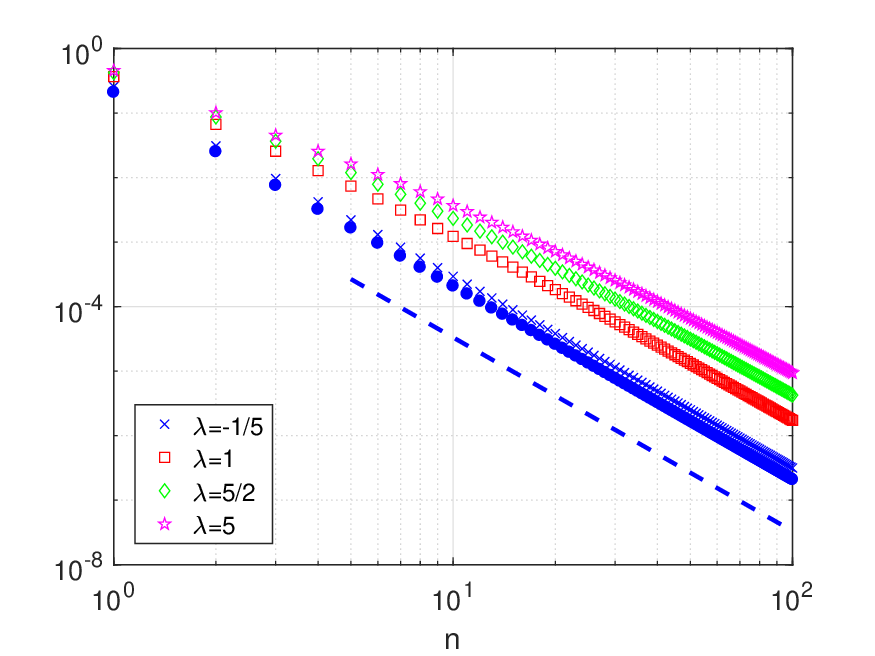}~
\includegraphics[width=7.5cm,height=6cm]{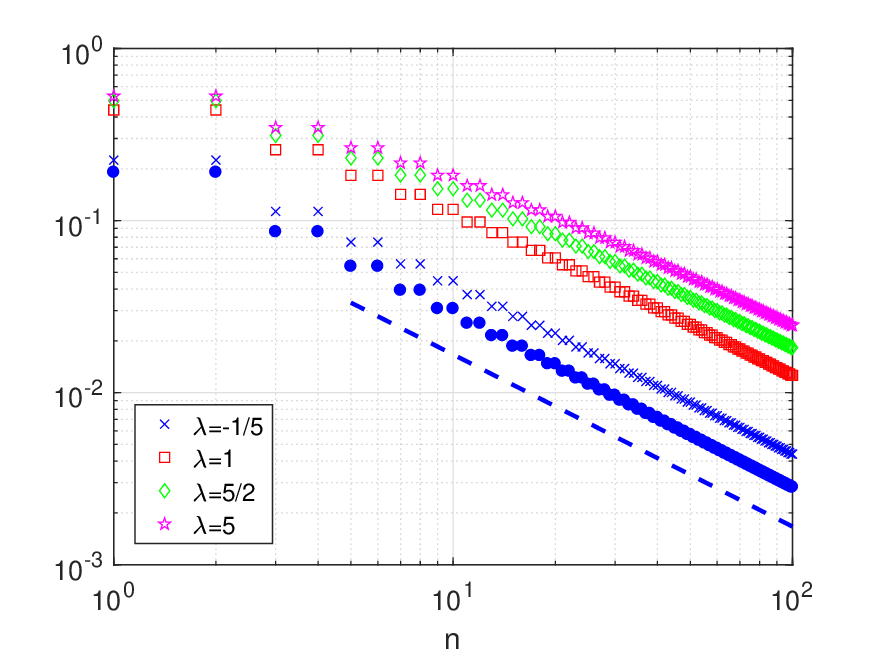}
\caption{The maximum errors of $\mathcal{B}_n(f)$ (dots) and
$S_n^{\lambda}(f)$ with four values of $\lambda$ for
$f(x)=(1+x)^{3/2}$ (left) and $f(x)=\arccos(x)$ (right). The dashed
line in the left panel is $O(n^{-3})$ and in the right panel is
$O(n^{-1})$.} \label{fig:ExamIX}
\end{figure}

In Figure \ref{fig:ExamIX} we illustrate the maximum errors of
$\mathcal{B}_n(f)$ and $S_n^{\lambda}(f)$ for $f(x)=(1+x)^{3/2}$ and
$f(x)=\arccos(x)$. It is easily seen that $\alpha=3/2$ for the
former and $\alpha=1/2$ for the latter. As expected, we observe that
the rate of convergence of $\mathcal{B}_n(f)$ is better than that of
$S_n^{\lambda}(f)$ by only constant factors. Moreover, we also see
that the maximum error of $S_n^{\lambda}(f)$ indeed deteriorates
slightly as $\lambda$ increases.

\subsection{An explanation of the error localization property}
For functions with an interior singularity, it has been observed in
Wang (2021) that the pointwise error of Legendre projections has the
error localization property, i.e., the error at the interior
singularity is obviously larger than the error away from the
singularity. However, a rigorous analysis of this observation is
still lacking. Here we restrict ourselves to the model function
\eqref{def:Model} and provide a theoretical explanation:
\begin{itemize}
\item In the case where $\theta\in(-1,1)$, we know from
\eqref{eq:PointwiseInt} that the convergence rate of
$S_n^{\lambda}(f)$ at each point $x\in(-1,\theta)\cup(\theta,1)$ is
faster than the convergence rate at $x=\theta$ as
$n\rightarrow\infty$. Moreover, the convergence rate of
$S_n^{\lambda}(f)$ at $x=\pm1$ is faster than the convergence rate
at all $x\in(-1,1)$ whenever $\lambda<0$ and is slower than the
convergence rate at $x\in(-1,\theta)\cup(\theta,1)$ whenever
$\lambda>0$.

\item In the case where $\theta=\pm1$, we know from
\eqref{eq:PointwiseEnd} that the convergence rate of
$S_n^{\lambda}(f)$ at each point $x\in(-1,1)$ is faster than the
convergence rate at $x=\theta$, especially when $\lambda$ is large.
Moeover, the convergence rate of $S_n^{\lambda}(f)$ at $x=-\theta$
is always faster than the convergence rate at $x=\theta$.
\end{itemize}
It is clear from these results that the error of $S_n^{\lambda}(f)$
at the singularity $x=\theta$ is obviously larger than the error at
$x\in(-1,\theta)\cup(\theta,1)$ for large $n$ and the maximum error
of $S_n^{\lambda}(f)$ is always attained at one of the critical
points, i.e., $x=\theta,\pm1$. This gives a clear explanation for
the error localization phenomenon of Gegenbauer projections.

\begin{figure}[ht]
\centering
\includegraphics[width=7.5cm,height=6cm]{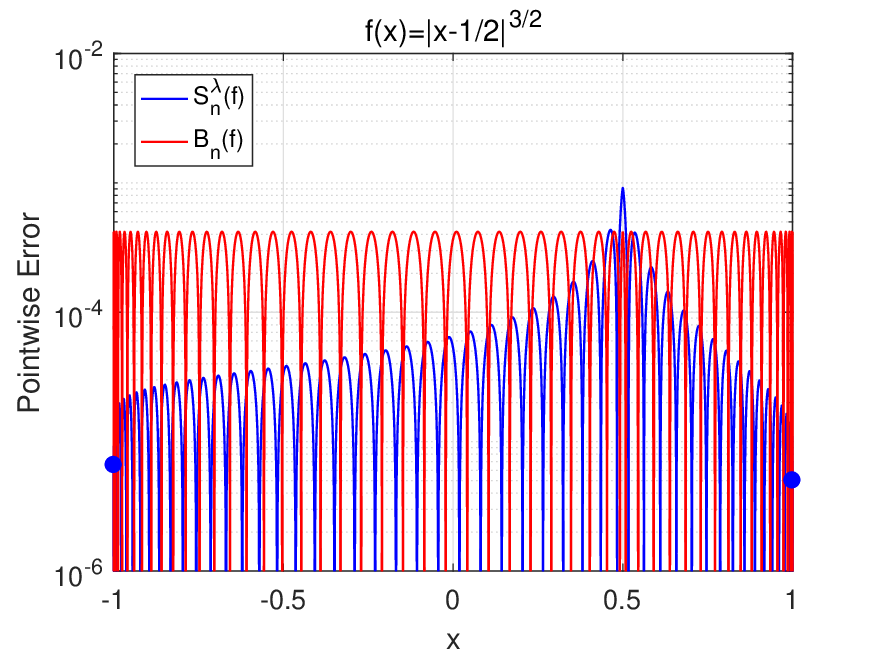}~
\includegraphics[width=7.5cm,height=6cm]{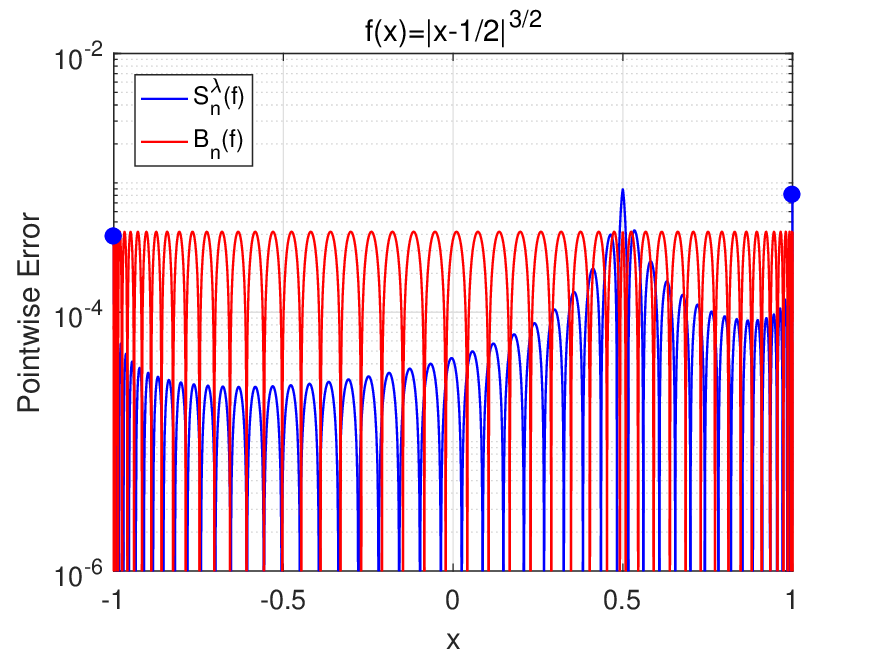}\\
\includegraphics[width=7.5cm,height=6cm]{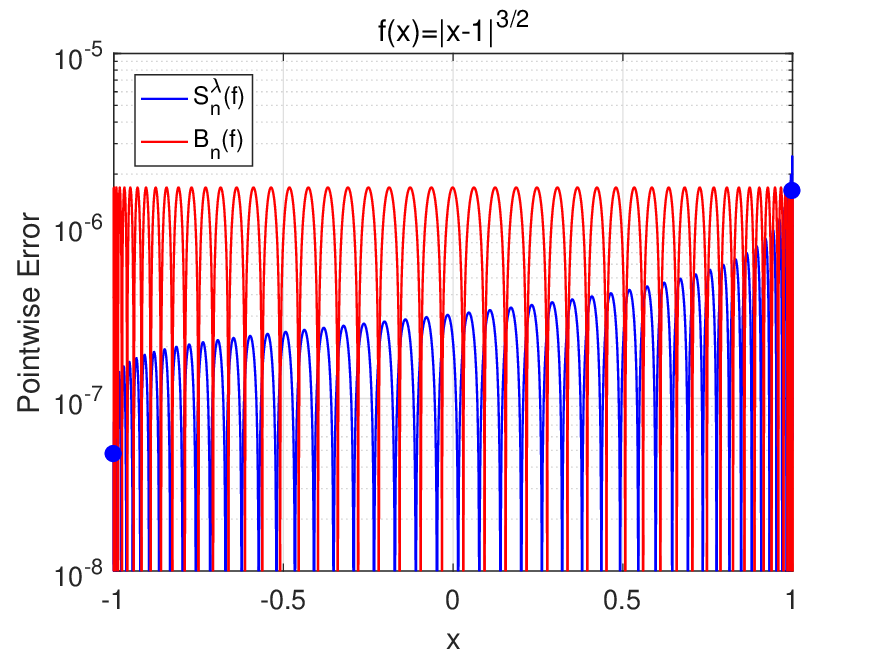}~
\includegraphics[width=7.5cm,height=6cm]{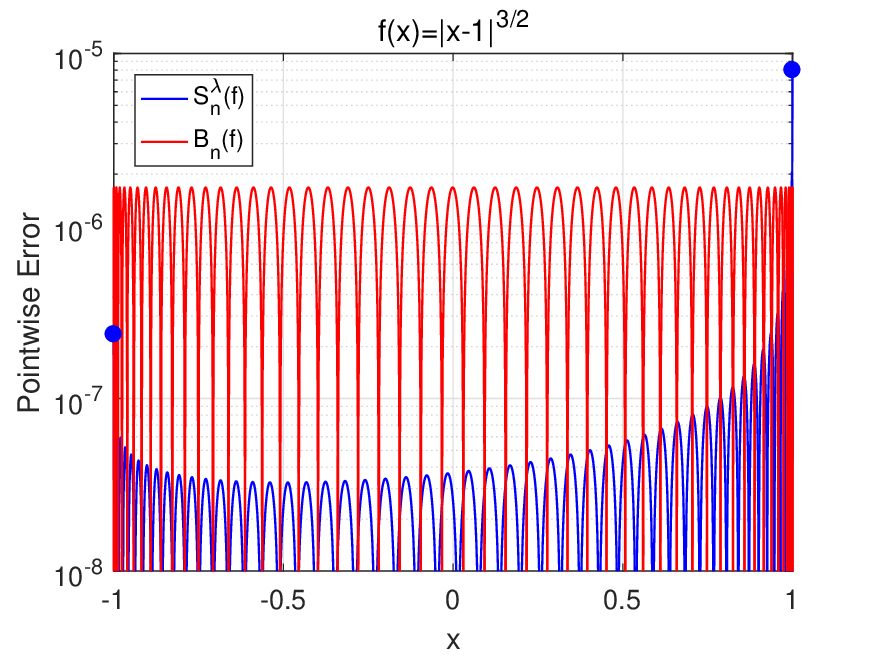}
\caption{Top row shows the pointwise errors of $\mathcal{B}_n(f)$
and $S_n^{\lambda}(f)$ for $\lambda=-1/4$ (left) and $\lambda=3/4$
(right). Bottom row shows the pointwise errors of $\mathcal{B}_n(f)$
and $S_n^{\lambda}(f)$ for $\lambda=-1/4$ (left) and $\lambda=1/2$
(right). Here we choose $n=50$ and these points indicate the errors
at $x=\pm1$.} \label{fig:ExamVII}
\end{figure}

\begin{remark}
Let $p_n^{L_1}(x)$ be the best polynomial approximation of degree
$n$ to $f$ in the $L_1$ norm. Very recently, it was shown in
Nakatsukasa \& Townsend (2021) that $p_n^{L_1}(x)$ also has the
error localization property, that is, the error of $p_n^{L_1}(x)$ is
obviously smaller than the error of $\mathcal{B}_n(f)$ except for a
set of small measure. We refer the reader to Nakatsukasa \& Townsend
(2021) for the discussion of the examples $f(x)=\sqrt{1-x^2}$ and
$f(x)=|x|$.
\end{remark}

On the other hand, we know from the equioscillation theorem that the
maximum error of $\mathcal{B}_n(f)$ is attained at least at $n+2$
points on $[-1,1]$ and the convergence rate of $\mathcal{B}_n(f)$ is
$O(n^{-\alpha})$ whenever $\theta\in(-1,1)$ and is $O(n^{-2\alpha})$
whenever $\theta=\pm1$. Hence, we can deduce that $S_n^{\lambda}(f)$
is actually more accurate than $\mathcal{B}_n(f)$ except in the
neighborhood of critical points. In Figure \ref{fig:ExamVII} we show
the pointwise errors of $S_n^{\lambda}(f)$ and $\mathcal{B}_n(f)$
for $\theta=1/2$ (top) and $\theta=1$ (bottom). Clearly, we observe
that numerical results are consistent with our analysis.

\section{Concluding remarks}\label{sec:conclusion}
In this work, we have compared the convergence behavior of
Gegenbauer projections $S_n^{\lambda}(f)$ and best approximations
$\mathcal{B}_n(f)$ and analyzed optimal rates of convergence of
Gegenbauer projections $S_n^{\lambda}(f)$ in the maximum norm. In
the case of analytic functions, we established some explicit error
bounds for $S_n^{\lambda}(f)$ in the maximum norm and proved that
these bounds are optimal in the sense that they can not be further
improved with respect to $n$. In the case of piecewise analytic
functions of class $C^{m-1}(\Omega)$ with $m\in\mathbb{N}$, we also
established optimal rates of convergence of $S_n^{\lambda}(f)$ in
the maximum norm. With these results, we showed that
$S_n^{\lambda}(f)$ and $\mathcal{B}_n(f)$ converge at the same rate
in the context of either $f$ is analytic and $\lambda\leq0$ or
$f\in{C}^{m-1}(\Omega)$ with $m\in\mathbb{N}$ is piecewise analytic
on $\Omega$ and $\lambda\leq1$. Otherwise, the rate of convergence
of $S_n^{\lambda}(f)$ is slower than that of $\mathcal{B}_n(f)$ by a
factor of $n^{\lambda}$ whenever $f$ is analytic and $\lambda>0$ and
by a factor of $n^{\lambda-1}$ whenever $f\in{C}^{m-1}(\Omega)$ is
piecewise analytic on $\Omega$ and $\lambda>1$. We also studied
optimal rates of convergence of Gegenbauer projections for functions
with algebraic singularities and we focused on the model function
$f(x)=|x-\theta|^{\alpha}$, where $\theta\in\Omega$ and $\alpha>0$
is not an even integer whenever $\theta\in(-1,1)$ and is not an
integer whenever $\theta=\pm1$. In the case $\theta\in(-1,1)$, we
showed that the maximum error of $S_n^{\lambda}(f)$ is attained at
one of the critical points, i.e., $x=\theta$ and $\pm1$, and the
rate of convergence of $S_n^{\lambda}(f)$ is the same as that of
$\mathcal{B}_n(f)$ for $\lambda\leq1$ and is slower than that of
$\mathcal{B}_n(f)$ by a factor of $n^{\lambda-1}$ for $\lambda>1$.
In the case $\theta=\pm1$, we show that the maximum error of
$S_n^{\lambda}(f)$ is attained at $x=\theta$ and both
$S_n^{\lambda}(f)$ and $\mathcal{B}_n(f)$ always converge at the
same rate for all $\lambda>-1/2$. We also provided an explanation
for the error localization property of Gegenbauer projections and
showed that Gegenbauer projections are actually more accurate than
best approximations except in the neighborhood of critical points.
All these findings were illustrated by numerical experiments.

We close this paper by clarifying the effect of the difference of
the size of Gegenbauer polynomials at the endpoints and in the
interior of $\Omega$ on the maximum error of Gegenbauer projections.
In the case where the singularity of the underlying function is
located at the interior of $\Omega$, by Theorem \ref{thm:InterRate}
we know that the difference of the size of Gegenbauer polynomials at
the endpoints and in the interior of $\Omega$ leads to the jump
phenomenon of the location of the maximum error of Gegenbauer
projections, as shown in Figure \ref{fig:ExamVI}. In this case, the
difference of the size of Gegenbauer polynomials at the endpoints
and in the interior of $\Omega$ accounts for the maximum error of
Gegenbauer polynomials. In the case where the singularity is located
at one of the endpoints, by Theorem \ref{thm:EndRate} we know that
the maximum error of Gegenbauer projections is always determined by
the error at the singularity and thus the difference of the size of
Gegenbauer polynomials at the endpoints and in the interior of
$\Omega$ has no effect on the maximum error of Gegenbauer
projections.

\section*{Acknowledgements}
This work was supported by National Natural Science Foundation of
China under grant number 11671160. The author thanks the two
anonymous reviewers for their helpful comments on the manuscript.

\appendix

\section{Proof of Lemma \ref{lem:IntSing}}
\begin{proof}
To show \eqref{eq:ExpFormula}, we follow the idea of Theorem 4.3 in
\cite{liu2019optimal} for Chebyshev coefficients. Let $m=\lfloor
\alpha \rfloor$ and $s=\alpha-m\in[0,1)$. Invoking the Rodrigues
formula \eqref{eq:Rodrigues} and using integration by parts $m+1$
times, we have for $k\geq m+1$ that
\begin{align}\label{eq:IntS1}
a_k^{\lambda} &= \frac{1}{h_k^{\lambda}} \prod_{j=0}^{m}
\frac{2(\lambda+j)}{(k-j)(k+2\lambda+j)} \int_{-1}^{1} f^{(m+1)}(x)
\omega_{\lambda+m+1}(x) C_{k-m-1}^{\lambda+m+1}(x) \mathrm{d}x \nonumber \\
&= \frac{1}{h_k^{\lambda}} \prod_{j=0}^{m}
\frac{2(\lambda+j)}{(k-j)(k+2\lambda+j)} \left[ \int_{-1}^{\theta}
f^{(m+1)}(x) \omega_{\lambda+m+1}(x)
C_{k-m-1}^{\lambda+m+1}(x) \mathrm{d}x \right. \nonumber \\
&~~~~~~~~~~ \left. + \int_{\theta}^{1} f^{(m+1)}(x)
\omega_{\lambda+m+1}(x) C_{k-m-1}^{\lambda+m+1}(x) \mathrm{d}x
\right].
\end{align}
We first consider the case $s=0$ (i.e., $\alpha=m$ is an odd
integer). In this case, direct calculations show that the $(m+1)$th
derivative of $f$ in the distributional sense is given by
$f^{(m+1)}(x)=2m!\delta(x-\theta)$, where $\delta(x)$ is the Dirac
delta function. Substitution of this into the first equality of
\eqref{eq:IntS1} gives
\begin{align}\label{eq:IntS11}
a_k^{\lambda} &= \frac{2 m!}{h_k^{\lambda}} \left[ \prod_{j=0}^{m}
\frac{2(\lambda+j)}{(k-j)(k+2\lambda+j)} \right]
\omega_{\lambda+m+1}(\theta) C_{k-m-1}^{\lambda+m+1}(\theta).
\end{align}
Combining \eqref{eq:IntS11}, \eqref{def:GegenPoly} and the symmetry
of Gegenbauer polynomials (i.e., $C_k^{\lambda}(-x)=(-1)^k
C_k^{\lambda}(x)$) gives the desired result \eqref{eq:ExpFormula}.
This proves the case $s=0$.

In the following, we consider the case $s\in(0,1)$. We consider to
derive explicit forms of these two integrals inside the square
bracket of \eqref{eq:IntS1}. For simplicity of notation, we denote
the former one by $J_1$ and the latter one by $J_2$. From
\cite[Equation~(3.12b)]{liu2019optimal}, we know that
\begin{align}\label{eq:IntS2}
\omega_{\lambda+m+1}(x) C_{k-m-1}^{\lambda+m+1}(x) &=
\frac{\Gamma(k+m+2\lambda+1)\Gamma(\lambda+m+\frac{3}{2})}{\Gamma(k-m)\Gamma(2m+2\lambda+2)
2^{s-1} \Gamma(\lambda+\alpha+\frac{1}{2})} \nonumber \\
&~~~~~ \times {}_{-1}\mathcal{I}_{x}^{1-s} \left\{
\omega_{\lambda+\alpha}(x) {}^{l}G_{k-\alpha}^{(\lambda+\alpha)}(x)
\right\},  
\end{align}
where ${}_{a}\mathcal{I}_{x}^{\nu}(\cdot)$ is the left fractional
integral of order $\nu$ and ${}^{l}G_{\nu}^{(\lambda)}(x)$ is the
left generalized Gegenbauer function of fractional degree $\nu$
defined by
\begin{align}
{}_{a}\mathcal{I}_{x}^{\nu}(f) = \frac{1}{\Gamma(\nu)} \int_{a}^{x}
\frac{f(t)}{(x-t)^{1-\nu}} \mathrm{d}t, \quad
{}^{l}G_{\nu}^{(\lambda)}(x) = (-1)^{\lfloor \nu \rfloor}
{}_2\mathrm{ F}_1\left[\begin{matrix} -\nu, & \nu+2\lambda
\\   \lambda + \frac{1}{2}  \hspace{-1cm} &\end{matrix} ;~ \frac{1+x}{2}
\right].  \nonumber
\end{align}
For $J_1$, using \eqref{eq:IntS2} and fractional integration by
part, we obtain
\begin{align}\label{eq:IntS3}
J_1 &=
\frac{\Gamma(k+m+2\lambda+1)\Gamma(\lambda+m+\frac{3}{2})}{\Gamma(k-m)\Gamma(2m+2\lambda+2)
2^{s-1} \Gamma(\lambda+\alpha+\frac{1}{2})}  \nonumber \\
&~~~~~~~~~~ \times \int_{-1}^{\theta} f^{(m+1)}(x)
{}_{-1}\mathcal{I}_{x}^{1-s} \left\{ \omega_{\lambda+\alpha}(x)
{}^{l}G_{k-\alpha}^{(\lambda+\alpha)}(x) \right\} \mathrm{d}x
\nonumber \\
&=
\frac{\Gamma(k+m+2\lambda+1)\Gamma(\lambda+m+\frac{3}{2})}{\Gamma(k-m)\Gamma(2m+2\lambda+2)
2^{s-1} \Gamma(\lambda+\alpha+\frac{1}{2})}  \nonumber \\
&~~~~~~~~~~ \times \int_{-1}^{\theta} \omega_{\lambda+\alpha}(x)
{}^{l}G_{k-\alpha}^{(\lambda+\alpha)}(x)
{}_{x}\mathcal{I}_{\theta}^{1-s} \left\{ f^{(m+1)}(x) \right\}
\mathrm{d}x,
\end{align}
where ${}_{x}\mathcal{I}_{\theta}^{\nu}(\cdot)$ is the right
fractional Riemann-Liouville integral of order $\nu$. For
$x\in(-1,\theta)$, a direction calculation shows that
${}_{x}\mathcal{I}_{\theta}^{1-s}\{ f^{(m+1)}\} = (-1)^{m+1}
\Gamma(\alpha+1)$. Moreover, using
\cite[Equation~(3.13b)]{liu2019optimal}, we have
\begin{align}
\omega_{\lambda+\alpha}(x) {}^{l}G_{k-\alpha}^{(\lambda+\alpha)}(x)
=
-\frac{\Gamma(\lambda+\alpha+\frac{1}{2})}{2\Gamma(\lambda+\alpha+\frac{3}{2})}
\frac{\mathrm{d}}{\mathrm{d}x} \left\{ \omega_{\lambda+\alpha+1}(x)
{}^{l}G_{k-\alpha-1}^{(\lambda+\alpha+1)}(x) \right\}, \nonumber
\end{align}
and therefore, we arrive at
\begin{align}\label{eq:IntS4}
J_1 &= (-1)^m
\frac{\Gamma(k+m+2\lambda+1)\Gamma(\lambda+m+\frac{3}{2})
\Gamma(\alpha+1)}{\Gamma(k-m)\Gamma(2m+2\lambda+2) 2^{s}
\Gamma(\lambda+\alpha+\frac{3}{2})}
\omega_{\lambda+\alpha+1}(\theta)
{}^{l}G_{k-\alpha-1}^{(\lambda+\alpha+1)}(\theta) \nonumber\\
&= (-1)^k \frac{\Gamma(k+m+2\lambda+1)\Gamma(\lambda+m+\frac{3}{2})
\Gamma(\alpha+1)}{\Gamma(k-m)\Gamma(2m+2\lambda+2) 2^{s}
\Gamma(\lambda+\alpha+\frac{3}{2})}
\omega_{\lambda+\alpha+1}(\theta) \nonumber \\
&~~~~~~ \times {}_2\mathrm{ F}_1\left[\begin{matrix} \alpha+1-k, &
k+2\lambda+\alpha+1
\\  \alpha + \lambda + \frac{3}{2} \hspace{-1cm} &\end{matrix} ;~ \frac{1+\theta}{2}
\right].
\end{align}
Using similar arguments, we can obtain
\begin{align}\label{eq:IntS5}
J_2 &= \frac{\Gamma(k+m+2\lambda+1)\Gamma(\lambda+m+\frac{3}{2})
\Gamma(\alpha+1)}{\Gamma(k-m)\Gamma(2m+2\lambda+2) 2^{s}
\Gamma(\lambda+\alpha+\frac{3}{2})}
\omega_{\lambda+\alpha+1}(\theta) \nonumber \\
&~~~~~~ \times {}_2\mathrm{ F}_1\left[\begin{matrix} \alpha+1-k, &
k+2\lambda+\alpha+1
\\ \alpha + \lambda + \frac{3}{2} \hspace{-1cm} &\end{matrix} ;~ \frac{1-\theta}{2}
\right].
\end{align}
Inserting \eqref{eq:IntS4} and \eqref{eq:IntS5} into
\eqref{eq:IntS1}, we obtain \eqref{eq:ExpFormula}.

As for \eqref{eq:AsymInt}, it follows from applying the asymptotic
expansion of Gauss hypergeometric function in
\cite[Equation~(4.7)]{paris2013asym} (with $\varepsilon=1$) to
\eqref{eq:ExpFormula}. This ends the proof.
\end{proof}

\end{document}